\documentclass[11pt]{article}
\usepackage{amssymb,amsmath}
\usepackage{amssymb} 
\usepackage{color}
\usepackage{graphicx}
\usepackage[vmargin=2.5cm, hmargin=2.5cm]{geometry}
\usepackage{enumerate}

\title{On canonical bases of a formal ${\mathbb K}$-algebra}
\author{Abdallah Assi
\footnote{Universit\'e d'Angers, Math\'ematiques,
49045 Angers ceded 01, France\newline
}}

\date{}

\newtheorem{teorema}{Theorem}
\newtheorem{proposicion}[teorema]{Proposition}
\newtheorem{lema}[teorema]{Lemma}

\newtheorem{corolario}[teorema]{Corollary}
\newtheorem{example}[teorema]{Example}

\newtheorem{definicion}[teorema]{Definition}
\newtheorem{nota}[teorema]{Remark}
\newenvironment{proof}{\noindent Proof.}

\newcommand{\CC}{{\bf F}}
\newcommand{\DD}{{\bf P}}
\newcommand{\KK}{{\mathbb K}}
\newcommand{\NN}{{\mathbb N}}
\newcommand{\RR}{{\mathbb R}}
\newcommand{\M}{\mathrm {M}}
\newcommand{\e}{\mathrm{exp}}

\begin{document}
\maketitle

\begin{abstract} We study canonical bases of a subalgebra  ${\bf A}={\mathbb K}[\![f_1,\dots,f_s]\!]\subseteq {\mathbb K}[\![x_1,\dots,x_n]\!]$
over a field ${\mathbb K}$, and we associate with ${\bf A}$  a fan called the canonical fan of $\bf A$. This generalizes the notion of
the standard fan  of an ideal.
\medskip

\end{abstract}

\section*{Introduction}

\medskip

\noindent Let $\KK$ be a field and let $f_1,\dots,f_s$ be nonzero elements of the 
ring $\CC={\mathbb K}[\![x_1,\dots,x_n]\!]$ of formal power series in $x_1,\dots,x_n$ over ${\mathbb K}$. Let
 ${\bf A}=\mathbb{K}[\![f_1,\dots,f_s]\!]$ be the ${\mathbb K}$-algebra generated
 by $f_1,\dots,f_s$. Set $U={{\mathbb R}_+^*}$  and let $a\in U^n$. If $a=(a_1,\cdots,a_n)$, then $a$ defines
a linear form on $\RR^n$ that maps   
 $\alpha=(\alpha_1,\dots,\alpha_n)\in {\RR}^n$ to the inner product 

$$
a\cdot\alpha=\sum_{i=1}^na_i\alpha_i
$$
 
\noindent  of $a$ with $\alpha$.  Let $\underline{x}=(x_1,\dots,x_n)$ and
let $f=\sum c_{\alpha}\underline{x}^{\alpha}$ be a nonzero element of ${\CC}$. We set
$\mathrm {Supp}(f)=\lbrace \alpha \mid  c_{\alpha}\not=0\rbrace$ and we call it the support of $f$. We set 

$$
\nu(f,a)={\rm min}\lbrace a\cdot\alpha \mid \alpha\in \mathrm{Supp}(f)\rbrace
$$

\noindent  and we call it the $a$-valuation of $f$. We set by convention $\nu(0,a)=+\infty$. Let 

$$
\mathrm{ in}(f,a)=\sum_{\alpha\in \mathrm{Supp}(f) \mid a\cdot\alpha=\nu(f,a)}c_{\alpha}\underline{x}^{\alpha}.
$$

\noindent  We call $\mathrm{in}(f,a)$ the $a$-initial form of $f$ (note that $\mathrm{in}(f,a)$ is a polynomial). 

\medskip

\noindent Let the notations be as above, and let $\prec$ be a well ordering on ${\mathbb N}^n$. We set $\mathrm {exp}(f,a)={\rm max}_{\prec}\mathrm {Supp}\\(\mathrm {in}(f,a))$
 and $\mathrm {M}(f,a)=c_{\mathrm {exp}(f,a)}\underline{x}^{\mathrm {exp}(f,a)}$. We set $\mathrm {in}({\bf A},a)=\mathbb{K}[\![\mathrm {in}(f,a), f\in {\bf A}\setminus\lbrace 0\rbrace]\!]$. We also set $\mathrm {M}({\bf A},a)={\mathbb K}[\![\mathrm {M}(f,a), f\in{\bf A}\setminus\lbrace 0\rbrace]\!]$. The set $\lbrace \mathrm {exp}(f,a)\mid f\in{\bf A}\setminus \lbrace 0\rbrace\rbrace$ is a subsemigroup of ${\mathbb N}^n$. We denote it by $\mathrm {exp}({\bf A},a)$. A set $S\subseteq {\bf A}$ is said to be an $a$-canonical basis of ${\bf A}$ if $\mathrm{exp}({\bf A},a)$ is generated by
 $\lbrace \mathrm{exp}(f,a) \mid f\in S\rbrace$.
\medskip

\noindent If  $a\in{\mathbb R}_+^n$, then $\mathrm {in}(f,a)$ may not be a polynomial, hence $\mathrm{exp}(f,a)$ is not 
well defined. If $a=(a_1,\dots,a_n)$ with $a_{i_1}=\cdots=a_{i_l}=0$, then we can avoid this difficulty in completing by the tangent cone order on $(x_{i_1},\dots,x_{i_l})$. We shall however consider elements in $U^n$ in order to avoid technical definitions and results. 

\medskip 

\noindent This paper has two objectives: we first give an algorithm that calculates, given $a\in U^n$ and a system of generators of ${\bf A}$, an $a$-canonical basis of ${\bf A}$, 
under the assumption that such a basis is finite, and we give criteria for $\mathrm{exp}({\bf A}, a)$ to be  finitely generated. Then we study the 
stability of $\mathrm{M}({\bf A},a)$ and $\mathrm {in}({\bf A},a)$ when $a$ varies in $U^n$. More precisely, we prove, {\it under some specific conditions}, that 
the sets $\lbrace \mathrm{M}({\bf A},a)\mid a\in U^n\rbrace$ and $\lbrace \mathrm{in}({\bf A},a)\mid a\in U^n\rbrace$ are  finite. Then we prove, given a finitely generated  semigroup $S$ of ${\mathbb N}^n$, that the set $E_S=\lbrace 
a\in U^n\mid \mathrm{exp}({\bf A},a)=S\rbrace$  is a union of convex polyhedral cones and the set of all the obtained $E_S$  define a fan of $U^n$. These results generalize those in \cite{mr} for ideals in ${\mathbb K}[x_1,\dots,x_n]$ and \cite{acg1}, \cite{acg2} for ideals in the ring of differential operators over ${\mathbb K}$.








\medskip

\noindent The paper is organized as follows. In Section 1 we recall the notion of canonical basis of ${\bf A}$ with respect to 
$a\in U^n$, and we give an algorithm that computes an $a$-canonical basis starting with a set of generators 
of ${\bf A}$ (see \cite{m} for the case $a=(1,\dots,1)$). In Section 2 we give criteria for an $a$-canonical basis to be finite. In Section 3 we 
study the possible connections between the set of $\M({\bf A},a), a\in U^n$, and we prove a finiteness  theorem for the set $\M({\bf A})$. In Section 4 we
 prove the existence of a fan associated with ${\bf A}$.

\medskip

\section{Computing canonical bases}

\medskip

\noindent Let the notations be as in the introduction. In particular ${\bf A}={\mathbb K}[\![f_1,\dots,f_s]\!]$ is 
a subalgebra of ${\CC}$ generated by $\lbrace f_1,\dots,f_s\rbrace\subseteq {\CC}$ and  $\prec$ is a total well ordering on $\NN^{n}$
compatible with sums. Let $a\in {U}^n$ and consider the total 
ordering  on $\NN^{n}$
defined by:

$$
 \begin{array}{ccl}\alpha <_a \alpha' &
\Leftrightarrow & \left\{ \begin{array}{l} a\cdot\alpha >
a\cdot\alpha' \\ \mbox{or} \\ a\cdot\alpha=
a\cdot\alpha' \mbox{ and } \alpha \prec \alpha'{ . }\end{array} \right.
\end{array}
 $$

\noindent The total ordering $<_a$ is compatible with sums in
$\NN^{n}$.  Let $f=\sum_{\alpha}c_{\alpha}\underline{x}^{\alpha}$ be a non zero element of ${\bf F}$.  With the notations of the introduction we have $\mathrm{exp}(f,a)=\mathrm{max}_{<_a}\mathrm{Supp}(f)=
\mathrm{max}_{\prec}\mathrm{Supp}(\mathrm{in}(f,a))$.
\medskip
 
\noindent We shall use sometimes the notations $\alpha \succ \beta$ for $\beta\prec \alpha$ and $\alpha >_a \beta$ for $\beta <_a \alpha$. We have the following:

\begin{lema}\label{Dixon} There doesn't exist  infinite sequences $(\alpha_k)_{k\geq 0}$ such that 

$$
\alpha_0 >_a \alpha_1 >_a \cdots >_a \alpha_k >_a\cdots 
$$

\noindent with $a \cdot\alpha_0=a\cdot\alpha_k$ for all $k\geq 1$.

\end{lema}

\begin{proof} This is a consequence of Dickson's Lemma,
since such a sequence satisfies $\alpha_0 \succ \alpha_1 \succ \cdots \succ \alpha_k  \succ\cdots$ $\blacksquare$
\end{proof}
\medskip

\noindent Let  $a\in U^n$. Then $a$ defines a graduation on ${\CC}: \CC=\sum_{d\geq 0}\CC_d$ where ${\bf  F}_d$ is the ${\mathbb K}$-vector space
 generated by $\underline{x}^{\alpha}, a\cdot\alpha=d$. Let $U_1={\mathbb Q}_+^*$. If $a\in U_1^n$, then,  given two indices $d_1<d_2$, the set of indices $d$ such that $d_1<d<d_2$ in the graduation is clearly finite. In particular we get the following:

\begin{lema}\label{Dixon2} Let $\alpha,\beta \in {\mathbb N}^n$ and let $a\in U_1^n$. If  $\alpha >_a \beta$, then there doesn't exist infinite sequences $(\alpha_k)_{k\geq 0}$ such that 

$$
\alpha>_a \alpha_0 >_a \alpha_1 >_a \cdots >_a \beta.
$$
\end{lema}

\begin{proof}  The set $\lbrace a\cdot\gamma\in{\mathbb R} \mid a\cdot\alpha >a\cdot\gamma > a\cdot\beta\rbrace$ is finite, hence  the result is a consequence of
 Lemma \ref{Dixon}. $\blacksquare$
\end{proof}

\begin{definicion}\label{homogeneous} Let $a\in U^n$ and let $f= \sum
c_{\alpha}\underline{x}^\alpha$ be a nonzero element 
of ${\CC}$. We say that $f$ is $a$-homogeneous if $f\in \CC_d$ for 
some $d$. This is equivalent to $\nu(\underline{x}^{\alpha},a)=\nu(f,a)$ for all $\alpha\in\mathrm {Supp}(f)$. Every nonzero element $f=\sum
c_{\alpha}\underline{x}^\alpha\in \CC$ decomposes into $a$-homogeneous 
elements, i.e. $f=\sum_{k\geq d}f_k$, with $f_d= \mathrm {in}(f,a)$ and  for all $k> d$, if $f_k\not=0$, then $f_k\in {\CC}_k$. 
\end{definicion}

\noindent The following notations will be used later in the paper.





\begin{definicion} \label{Homogeneous} Let $a\in U^n$ and let ${\bf H}$ be a subalgebra of ${\CC}$. We say that $\bf H$ is
 $a$-homogeneous if it can be generated by $a$-homogeneous elements of ${\CC}$.
\end{definicion}

\begin{definicion} Let $a\in U^n$ and  suppose that $\mathrm{exp}({\bf A},a)$ is finitely generated. Let $\lbrace g_1,\dots,g_r\rbrace\subseteq {\bf A}$. We say 
that  $\lbrace g_1,\dots,g_r\rbrace$ is an $a$-canonical basis
 of ${\bf A}$ if $\M({\bf A},a)={\mathbb K}[\![{\mathrm M}(g_1,a),\dots,{\mathrm M}(g_r,a)]\!]$. Clearly 
$\lbrace g_1,\dots,g_r\rbrace$ is an $a$-canonical basis of ${\bf A}$ if and only if $\mathrm{exp}({\bf A},a)$ is 
generated by $\lbrace \mathrm{exp}(g_1,a), \dots,$\\
$\mathrm{exp}(g_r,a)\rbrace$. In this case we 
write $\mathrm{exp}({\bf A},a)=\langle \mathrm{exp}(g_1,a),\dots,\mathrm{exp}(g_r,a)\rangle$.
\end{definicion}

\noindent  A finite  $a$-canonical basis $\lbrace g_1,\dots,g_r\rbrace$
of $\bf A$ is said to be minimal if  $\lbrace \M(g_1,a),\dots,\M(g_r,a)\rbrace$ is a minimal
set of generators of $\M({\bf A},a)$.  It is said to be reduced if the following conditions
are satisfied:

i) $\lbrace g_1,\dots,g_r\rbrace$ is minimal.

ii) For all $1\leq i\leq r, c_{{\rm exp}(g_i,a)}=1$.

iii) For all $1\leq i\leq r$, if $g_i-\M(g_i,a)\not= 0$, then   $\underline{x}^{\underline{\alpha}}\notin {\mathbb K}[\![\M({g_1,a}),\dots, \M(g_r,a)]\!]$ for all    
$\alpha\in$ Supp$(g_i-\M(g_i,a))$.

\medskip

\begin{lema} \label{reduced} If $\mathrm{exp}({\bf A},a)$ is finitely generated and if an $a$-reduced canonical basis exists,  then it is unique. We write $e({\bf A},a)$ for its cardinality.
 \end{lema}

\begin{proof} Let  $F=\lbrace g_1,\dots,g_r\rbrace$ and $G=\lbrace g'_1,\dots,g'_t\rbrace$ be two
$a$-reduced canonical bases of $\bf A$. Let $i=1$. Since
 $\M(g_1,a)\in {\mathbb K}[\![\M({g'_1,a}),\dots, \M(g'_t,a)]\!]$,  
we have $\M(g_1,a)=\M(g'_1,a)^{l_1}\dots{\M(g'_t,a)}^{l_t}$ for
 some $l_1,\dots,l_t\in {\mathbb N}$. Every $\M(g'_i,a), i\in\lbrace 1,\dots,t\rbrace$, is in ${\mathbb K}[\![\M(g_1,a),\dots,\M(g_r,a)]\!]$, 
then the equation above is possible only if $\M(g_1,a)=\M(g'_{k_1},a)$ for some $k_1\in\lbrace 1,\dots,t\rbrace$. This 
gives an injective map from $\lbrace \M(g_1,a),\dots,\M(g_r,a)\rbrace$ to $\lbrace \M(g'_1,a),\dots,\M(g'_t,a)\rbrace$. We 
construct in the same way an injective map from $\lbrace \M(g'_1,a),\cdots,\M(g'_t,a)\rbrace$ to  $\lbrace \M(g_1,a),\dots,\M(g_r,a)\rbrace$. 
Hence $r=t$ and both sets are equal. Suppose, without loss of generality, that $\M(g_i,a)=\M(g'_i,a)$ for all $i\in\lbrace 1,\dots,r\rbrace$. 
If $g_i\not= g'_i$, then $\M(g_i-g'_i)\in \M({\bf A},a)$ because $g_i-g'_i\in{\bf A}$. This contradicts iii). $\blacksquare$
\end{proof}

\medskip

 \noindent We now recall the division process for algebras in ${\bf F}$ (see \cite{m} for the case $a=(1,\dots,1)$ and \cite{agm} for $n=1$).

\medskip

\begin{teorema}\label{division} Let $a\in U_1^n$ and let $\lbrace F_1,\dots,F_s\rbrace\subseteq \CC$. Let $F$ be a nonzero element of
$\CC$. There exist  $H\in {\mathbb K}[\![F_1,\dots,F_s]\!]$ and $R\in {\CC}$ such that the following conditions hold:

\begin{enumerate}

 \item $F=H+R$ and $\e(F,a)=\mathrm{max}(\e(H,a),\e(R,a))$ as long as $H\not=0$ and $R\not=0$.

\item  $H$ admits an expression of the form $H=\sum_{\alpha}c_{\alpha}\prod_{i=1}^sF_i^{\alpha_i}$ with 
$\e(F,a)\geq_a \e(\prod_{i=1}^sF_i^{\alpha_i},a)$ for all $\alpha$ with $c_{\alpha}\not=0$. Moreover, if $R=0$, then 
there exists a unique $\alpha^0$ such that $\e(F,a)= \e(\prod_{i=1}^sF_i^{\alpha^0_i},a)$.

\item  Let $R=\sum_{\beta}b_{\beta}{\underline{x}}^{\beta}$. If $R\not=0$, then for 
all $\beta\in\mathrm {Supp}(R), \underline{x}^{\beta}\notin{\mathbb{K}}[\![\M(F_1,a),\dots,\M(F_s,a)]\!]$ $\mathrm{(}$or equivalently 
$\beta\notin\langle \mathrm{exp}(F_1),\dots,\mathrm{exp}(F_s)\rangle$$\mathrm{)}$. Moreover, $\e(F,a)\geq \beta$.

\end{enumerate}

\noindent  We set   $R=R_a(F,\lbrace F_1,\dots,F_s\rbrace)$ and we say that $R$  is an $a$-remainder of the division of $F$ with respect to $\lbrace F_1,\dots,F_s\rbrace$.
\end{teorema} 

\begin{proof} We define the sequences
$(F^k)_{k\geq 0},(h^k)_{k\geq 0}$, $(r^k)_{k\geq 0}$ in $\CC$ by $F^0=F, h^0=r^0=0$ and $\forall k\geq
0$~:
\medskip

{\it (i)} If $F^k\not=0$ and if $\M(F^k,a)\in {\mathbb K}[\![\M(F_1,a),\dots,\M(F_s,a)]\!]$, write
$\M(F^k,a)=c_{\alpha}\prod_{i=1}^s\M(F_i,a)^{\alpha_i}$. We set 

$$
F^{k+1}=F^k- c_{\alpha}\prod_{i=1}^sF_i^{\alpha_i},\quad 
h^{k+1}=h^k+c_{\alpha}\prod_{i=1}^sF_i^{\alpha_i}, \quad r^{k+1}=r^k.
$$

{\it (ii)} If $F^k\not=0$ and if $\M(F^k,a)\notin {\mathbb K}[\![\M(F_1,a),\dots,\M(F_s,a)]\!]$, we set

$$
F^{k+1}=F^k-{\mathrm M}(F^{k},a),\quad
h^{k+1}=h^k,\quad r^{k+1}=r^k+{\mathrm M}(F^k,a)
$$

\noindent in such a way that for 
all $k\geq 0$, $\mathrm {exp}(F^k,a)>_a \mathrm {exp}(F^{k+1},a)$ 
and $F=F^{k+1}+h^{k+1}+r^{k+1}$. If $F^l=0$ for 
some $l\geq 1$, we set $H=h^l$ and $R=r^l$. We  easily verify that $H,R$ 
satisfy conditions (1) to (3). Suppose that $\lbrace F^{k}\mid  k\geq 0\rbrace$ is an infinite set. Note
 that, by Lemma \ref{Dixon}, given $k\geq 1$, if $F^k\not=0$, then there exists $k_1 >k$ such 
that $\nu(F^k,a)<\nu(F^{k_1},a)$. Hence,  there exists a subsequence $(F^{j_l})_{l\geq 1}$ such 
that $\nu(F^{j_1},a)<\nu(F^{j_2},a)<\cdots$. In particular, if we set 
$G=\lim_{k\to+\infty}F^k, H=\lim_{k\to+\infty}h^k$, and  $R=\lim_{k\to+\infty}r^k$, 
then $G=0$, $F=H+R$, and $H,R$ satisfy conditions (1) to (3). This completes the proof. $\blacksquare$
\end{proof}

\medskip
\begin{nota} The element  $R=R_a(F,\lbrace F_1,\dots,F_s\rbrace)$ of Theorem \ref{division} is not necessarily unique. For example,
if $a=1$, $F_1=x^4, F_2=x^6+x^7$, and $F=x^{12}$, then 
$\mathrm{M}(F,a)=x^{12}\in {\mathbb K}[\![\mathrm{M}(F_1,a),\mathrm{M}(F_2,a)]\!]= {\mathbb K}[\![x^4,x^6]\!]$. 
But $x^{12}=(x^4)^3=(x^6)^2$. If we divide first by $F_1$, we get $F-F_1^3=0$, hence $F=F_1^3+0$, and $R_a(F,\lbrace F_1,F_2\rbrace)=0$. If we divide first by 
$F_2$, we get $F-F_2^2=-2x^{13}-x^{14}$. As $x^{14}=(x^4)^2x^6$, we have 
$F-F_2^2+F_1^2F_2=-2x^{13}+x^{15}$, hence $F=F_2^2-F_1^2F_2-2x^{13}+x^{15}$, and $R_a(F_1,\lbrace F_1,F_2\rbrace)=-2x^{13}+x^{15}$. 

\end{nota} 

\medskip

\noindent The following lemma shows that $R_a(F,\lbrace F_1,\dots,F_s\rbrace)$ becomes unique if we divide by an $a$-canonical basis.

\begin{lema} Let $a\in U_1^n$ and let $\lbrace F_1,\dots,F_s\rbrace\subseteq \CC$. Suppose that $\lbrace F_1,\dots, F_s\rbrace$ is  an 
$a$-canonical basis of ${\mathbb K}[\![F_1,\dots,F_s]\!]$. If $F$ is a nonzero
 element of $\CC$, then $R_a(F,\lbrace F_1,\dots,F_s\rbrace)$ 
is unique.
\end{lema}

\begin{proof} Let the notations be as in Theorem \ref{division}  and write 
$F=H_1+R_1=H_2+R_2$  where $H_i,R_i, i=1,2$ satisfy conditions (1) to (3) of the theorem. We have
$R_1-R_2=H_2-H_1$. Clearly $R_1-R_2\in {\mathbb K}[\![F_1,\dots,F_s]\!]$. If $R_1-R_2\not=0$, then, by 
condition  (3), $\mathrm{M}(R_1-R_2,a)\notin  {\mathbb K}[\![\mathrm{M}(F_1,a),\dots,\mathrm{M}(F_s,a)]\!]$.
This is a contradiction because $\lbrace F_1,\dots,F_s\rbrace$ is an $a$-canonical basis of ${\mathbb K}[\![F_1,\dots,F_s]\!]$. $\blacksquare$
\end{proof}

\medskip

\noindent Suppose that $\lbrace f_1,\dots,f_s\rbrace$ is an $a$-canonical 
basis of $\bf A$. If
M$(f_i,a)\in {\mathbb K}[\![{\rm M}(f_j,a), j\not= i]\!]$ for some $1\leq i\leq s$,  then obviously
$\lbrace f_j \mid j\not=i\rbrace$ is also an $a$-canonical basis of $\bf A$. Consequently we can get this way a minimal $a$-canonical basis of ${\bf A}$. Assume  that  
$\lbrace f_1,\dots,f_s\rbrace$ is minimal and let $1\leq i\leq s$. 
Dividing  $f=f_i-\M(f_i,a)$ by $\lbrace f_1,\dots,f_s\rbrace$, and 
replacing $f_i$ by $\M(f_i,a)+R_a(f,\lbrace f_1,\dots,f_s\rbrace)$, we obtain  an $a$-reduced canonical basis of ${\bf A}$. 

\medskip 

\noindent The next proposition gives a criterion for a finite set of ${\bf A}$ to be an $a$-canonical basis of ${\bf A}$.

\begin{proposicion}\label{criterion} Let $a\in U_1^n$. The set $\lbrace f_1,\dots,f_s\rbrace\subseteq {\bf A}$ is an $a$-canonical basis of $\bf A$ if and only if 
$R_a(f,\lbrace f_1,\dots,f_s\rbrace)=0$ for all $f\in {\bf A}$.
\end{proposicion}

\begin{proof} Suppose that  $\lbrace f_1,\dots,f_s\rbrace$ is an $a$-canonical
 basis of $\bf A$ and let $f\in{\bf A}$. Let $R=R_a(f,\lbrace f_1,\dots$\\$,f_s\rbrace)$. If $R\not=0$, then
 $\M(R,a)\notin \M({\bf A},a)$. This is a contradiction because $R\in{\bf A}$. Conversely, suppose 
that $R_a(f,\lbrace f_1,\dots,f_s\rbrace)=0$ for all $f\in {\bf A}$ and 
let $F\in{\bf A}$. If $\M(F,a)\notin {\mathbb K}[\![\M(f_1,a),\dots,\M(f_s,a)]\!]$, 
then $\M(F,a)$ is  a monomial of $R_a(F,\lbrace f_1,\dots, f_s\rbrace)$, 
which is $0$. This is a contradiction. $\blacksquare$
\end{proof}

\medskip

\noindent The criterion given in Proposition \ref{criterion} is not 
effective since we have to divide infinitely many  elements of ${\CC}$. In 
the following we shall see that it is enough to divide a finite number of elements. 

\medskip

\noindent Let $\lbrace f_1,\dots,f_s\rbrace\subseteq \CC$ and let $\phi:{\mathbb K}[X_1,\dots,X_s]\to {\mathbb K}[\M(f_1,a),\dots,\M(f_s,a)]$ be the morphism 
of rings defined by  $\phi(X_i)=\M(f_i,a)$ for all $1\leq i\leq s$. We have the following.

\begin{lema} \label{binomial}The ideal $\mathrm {Ker}(\phi)$ is a binomial ideal, i.e., it can be generated by binomials.

\end{lema}

\begin{proof} See \cite{CLS}, Proposition 1.1.9., for example. $\blacksquare$ 
\end{proof}




\medskip

\noindent Let $\bar{S'}_1,\dots,\bar{S'}_{m_1}$ be a system of generators of Ker$(\phi)$, and assume, by Lemma \ref{binomial}, that 
$\bar{S'}_1,\dots,\bar{S'}_{m_1}$ are binomials in ${\mathbb K}[X_1,\dots,X_s]$. Assume that $f_1,\dots,f_s$ are monic with 
respect to $<_a$. For all $1\leq i\leq m_1$, we can write $\bar{S'}_i(X_1,\dots,X_s)=X_1^{\alpha_1^i}\cdots X_s^{\alpha_s^i}- X_1^{\beta_1^i}\cdots X_s^{\beta_s^i}$. Let 
$S'_i=\bar{S'}_i(f_1,\dots,f_s)=f_1^{\alpha_1}\dots f_s^{\alpha_s}-f_1^{\beta_1}\dots f_s^{\beta_s}
$. We have  $\mathrm {exp}(S'_i,a)<_a \mathrm {exp}(f_1^{\alpha_1}\dots f_s^{\alpha_s},a)=\mathrm {exp}(f_1^{\beta_1}\dots f_s^{\beta_s},a)$
(note that we only have $\nu(S'_i,a)\geq \nu(f_1^{\alpha_1}\dots f_s^{\alpha_s},a)=\nu(f_1^{\beta_1}\dots f_s^{\beta_s},a)$).


\medskip

\noindent Next we shall define the notion of $S$-polynomials. Let to this end $L(X_i)=\nu(f_i,a)$ for all $i\in\lbrace 1,\dots,s\rbrace$. Then $L$ defines a filtration on $\KK[X_1,\dots,X_s]$. We write abusely $L(\alpha)=L(\underline{X}^{\alpha})=\sum_{i=1}^s\alpha_i\nu(f_i,a)$ for all $\alpha\in{\mathbb N}^s$.  Given $0\not=F=\sum_{\alpha}c_{\alpha}\underline{X}^{\alpha}\in {\mathbb K}[X_1,\dots,X_n]$, we write $F=F_{d_1}+\dots+F_{d_l}$ where $d_1<d_2<\dots <d_l$ and $F_{d_i}$ is $L$-homogeneous. We set $L(F)=d_1$. Consider on $\NN^s$ the following total ordering: $\alpha<_g \beta$ if and only if either $L(\alpha)>L(\beta)$, or $L(\alpha)=L(\beta))$ and $\alpha \prec_g \beta$, where $\prec_g$ is a well ordering on $\NN^s$. Given $0\not=F=\sum_{\alpha}c_{\alpha}\underline{X}^{\alpha}$, the leading exponent of $F$ with respect to $<_g$ is denoted 
$\mathrm{Exp}(F)$, and the leading monomial $c_{\mathrm{Exp}}(F)\underline{X}^{\mathrm{Exp}(F)}$ is denoted $M(F)$. Given $i\in\lbrace 1,\dots,m_1\rbrace$, if $\bar{S'}_i=\underline{X}^{\alpha}-\underline{X}^{\beta}$, then $\sum_{i=}^s\alpha_i\nu(f_i,a)=\sum_{i=}^s\beta_i\nu(f_i,a)$, whence $\bar{S'}_i$ is $L$-homogeneous , and
$\mathrm{Exp}(\bar{S'}_i)$ is just the leading exponent of $S'_i$ with respect to $\prec_g$.

\begin{definicion} Let the notations be as above, and let $\bar{S}_1,\dots,\bar{S}_m$ be a minimal reduced Grobner basis of $\mathrm{Ker}(\phi)$ with respect to $\prec_g$. We call $\lbrace \bar{S}_1,\dots,\bar{S}_m\rbrace$ the set of $S$-polynomials  associated with $\lbrace f_1,\dots,f_s\rbrace$, and we recall that $\bar{S}_i$ is a binomial for all $i\in\lbrace 1,\dots,m\rbrace$.
\end{definicion}


\begin{proposicion}\label{S-polynomials} Let $a\in U_1^n$ and let  $\lbrace f_1,\dots,f_s\rbrace\subseteq {\bf A}$. Let $\lbrace \bar{S}_1,\dots,\bar{S}_m\rbrace$ be the set of $S$-polynomials  associated with $\lbrace f_1,\dots,f_s\rbrace$. For all $i\in \lbrace 1,\dots, m\rbrace$, we set $S_i=\bar{S}_i(f_1,\dots,f_s)$.  The following conditions are equivalent:

\begin{enumerate}
    \item The set $\lbrace f_1,\dots,f_s\rbrace$ is an $a$-canonical basis of ${\bf A}$.
    
    \item For all $i\in\lbrace 1,\dots, m\rbrace, R_a(S_i,\lbrace f_1,\dots,f_s\rbrace)=0$.
\end{enumerate}

\end{proposicion}

\begin{proof} (1) implies (2) by Proposition \ref{criterion}.

\medskip

(2) $\Longrightarrow$ (1): We shall prove that $R_a(f,\lbrace f_1,\dots,f_s\rbrace)=0$ for all $f\in{\bf A}$. Let $f$ be 
a nonzero element of ${\bf A}$ and let $R=R_a(f,\lbrace f_1,\dots,f_s\rbrace)$. By Theorem \ref{division}, $f=H+R$ and $H\in{\bf A}$, 
hence $R\in{\bf A}$. Consequently, if $R\not=0$, 
then $\M(R,a)\in \M({\bf A},a)$. Write

$$
R=\sum_{\theta}c_{\theta}f_1^{\theta_1}\cdots f_s^{\theta_s}
$$

\noindent and let $E=\lbrace \sum_{i=1}^s \theta_i\mathrm{exp}(f_i,a)=\mathrm{exp}(f_1^{\theta_1}\dots f_s^{\theta_s} \mid c_{\theta}\not=0\rbrace$. By  a similar argument to the one used in Lemma \ref{Dixon2}, the set $E_1=\lbrace \beta\in E \mid \beta  \geq_a \mathrm{exp}(R,a)\rbrace$ is finite. In particular,
$\alpha=\mathrm {max}_{\theta,c_{\theta}\not=0}(\mathrm {exp}(f_1^{\theta_1}\cdots f_s^{\theta_s},a))=\mathrm{max}_{<_a} E_1$ is well defined . Let $S= \langle \mathrm {exp}(f_1,a),\dots,\mathrm{exp}(f_s,a)\rangle$. By condition (2) of Theorem \ref{division},
$\mathrm {exp}(R,a)\notin S$. On the other hand,  $\mathrm{exp}(R,a)\leq_a \alpha\in  S$. Consequently 
$\mathrm {exp}(R,a)<_a \alpha$. \\ Let  $\lbrace {\theta}^1,\dots,{\theta}^l\rbrace$ be the set of elements such that $\alpha= {\rm exp}(f_1^{\theta^i_1}\cdots  f_s^{\theta^i_s},a)$ for all
 $i\in\{1,\dots, l\}$ (such a set is clearly finite). If $\sum_{i=1}^lc_{{\theta}^i}\M(f_1^{\theta^i_1}\cdots  f_s^{\theta^i_s},a)\not=0$, 
then $\mathrm{exp}(R,a)\in S$, which 
is a contradiction. It follows that  $\sum_{i=1}^lc_{{\theta}^i}\M(f_1^{\theta^i_1}\cdots  f_s^{\theta^i_s},a)=0$, and consequently 
 the polynomial $F=\sum_{i=1}^lc_{\underline{\theta}^i}X_1^{\theta^i_1}\cdots  X_s^{\theta^i_s}$ is an element of  $\mathrm{Ker}(\phi)$. In particular,  as $\lbrace \bar{S}_1,\dots,\bar{S}_m\rbrace$ is a Grobner basis of $\mathrm{Ker}(f)$, dividing $F$ by this set, we can write
 
 $$
F= \sum_{i=1}^lc_{\underline{\theta}^i}X_1^{\theta^i_1}\cdots  X_s^{\theta^i_s}=\sum_{k=1}^m\lambda_k\bar{S}_k
 $$
 
 
\noindent  with $\lambda_k\in \KK[X_1,\dots,X_s]$ for all $k\in\{1,\dots,m\}$, and $\mathrm{Exp}(F)=\mathrm{max}_{k}\mathrm{Exp}(\lambda_k\bar{S}_k)$.  Moreover, as $F$ is $L$-homogeneous, it follows that $\lambda_k$ is $L$-homogeneous and 
$L(\lambda_k)=L(F)-L(\bar{S}_k)$ for all $k\in\{1,\dots,m\}$. As

 $$
 \sum_{i=1}^lc_{\underline{\theta}^i}f_1^{\theta^i_1}\cdots  f_s^{\theta^i_s}=\sum_{k=1}^m\lambda_k(f_1,\dots,f_s)S_k
 $$
 
\noindent for all $i\in\lbrace 1,\dots,l\rbrace, \mathrm{exp}(f_1^{\theta^i_1}\cdots  f_s^{\theta^i_s},a)=\alpha$ (*),  and consequently, 
$\displaystyle{a\cdot\alpha}=\sum_{j=1}^s\theta_j^i\nu(f_j,a)=L(X_1^{\theta_1^i}\ldots X_s^{\theta_s^i})$. We claim that for all nonzero monomial  $c_{\gamma}X_1^{\gamma^k_1}\ldots X_s^{\gamma^k_s}$ of $\lambda_k$, if $\bar{S}_k=X_1^{\rho^k_1}\ldots X_s^{\rho^k_s}-X_1^{\xi^k_1}\ldots X_s^{\xi^k_s}$, then

$$
\alpha=\mathrm{exp}((f_1^{\gamma^k_1}\ldots f_s^{\gamma^k_s})(f_1^{\rho^k_1}\ldots f_s^{\rho^k_s}),a)=\mathrm{exp}((f_1^{\gamma^k_1}\ldots f_s^{\gamma^k_s})(f_1^{\xi^k_1}\ldots f_s^{\xi^k_s}),a).
$$

\noindent We shall use to this end the division of $F$ by $\lbrace \bar{S}_1,\dots,\bar{S}_m\rbrace$. Suppose, without loss of generality, that $\mathrm{Exp}(F)=(\theta_1^1,\dots,\theta_s^1)=\mathrm{Exp}(\lambda_1\bar{S}_1)$, and let $\mathrm{M}(\lambda_1)=cX_1^{\gamma_1}\ldots X_s^{\gamma_s}, \mathrm{M}(\bar{S}_1)=X_1^{\rho_1}\ldots X_s^{\rho_s}$, and $\bar{S}_k=X_1^{\rho_1}\ldots X_s^{\rho_s}-X_1^{\xi_1}\ldots X_s^{\xi_s}$. 
Finally let $F^1=F-\mathrm{M}(\lambda_1)S_1$. As 

$$
c_{\theta^1}X_1^{\theta^1_1}\ldots X_s^{\theta^1_s}=cX_1^{\gamma_1}\ldots X_s^{\gamma_s}X_1^{\rho_1}\ldots X_s^{\rho_s}
$$

\noindent we get

$$
c_{\theta^1}f_1^{\theta^1_1}\ldots f_s^{\theta^1_s}=cf_1^{\gamma_1}\ldots f_s^{\gamma_s}f_1^{\rho_1}\ldots f_s^{\rho_s}.
$$

\noindent Whence $\alpha=\mathrm{exp}(f_1^{\theta^1_1}\ldots f_s^{\theta^1_s})=\mathrm{exp}(f_1^{\gamma_1}\ldots f_s^{\gamma_s})+\mathrm{exp}(f_1^{\rho_1}\ldots f_s^{\rho_s})$, and by the definition of $S$-polynomials, $\alpha=\mathrm{exp}(f_1^{\theta^1_1}\ldots f_s^{\theta^1_s})=\mathrm{exp}(f_1^{\gamma_1}\ldots f_s^{\gamma_s})+\mathrm{exp}(f_1^{\xi_1}\ldots f_s^{\xi_s})$. As $F^1$ satisfies the condition (*), we get our claim by an easy induction.  Now we shall prove that for all $k\in\lbrace 1,\dots,m\rbrace$, we have $\mathrm{exp}(\lambda_k(f_1,\dots,f_s)S_k,a)<_a \alpha$. Let to this end $cX_1^{\delta_1}\ldots X_s^{\delta_s}$ be a nonzero monomial of $\lambda_k$, and write $\bar{S}_k=X_1^{\zeta_1}\ldots X_s^{\zeta_s}-X_1^{\tau_1}\ldots X_s^{\tau_s}$. In particular $S_k=f_1^{\zeta_1}\ldots f_s^{\zeta_s}-f_1^{\tau_1}\ldots f_s^{\tau_s}$. As  

\medskip

$$
\alpha=\mathrm{exp}((cf_1^{\delta_1}\ldots f_s^{\delta_s})(f_1^{\zeta_1}\ldots f_s^{\zeta_s}),a)=\mathrm{exp}((cf_1^{\delta_1}\ldots f_s^{\delta_s})(f_1^{\tau_1}\ldots f_s^{\tau_s}),a)
$$

\noindent  and $\mathrm{exp}(S_k,a)<_a\mathrm{exp}(f_1^{\zeta_1}\ldots f_s^{\zeta_s},a)=\mathrm{exp}(f_1^{\tau_1}\ldots f_s^{\tau_s},a)$, we get

$$
\mathrm{exp}(cf_1^{\delta_1}\ldots f_s^{\delta_s}S_k),a)<_a\mathrm{exp}(cf_1^{\delta_1}\ldots f_s^{\delta_s})(f_1^{\zeta_1}\ldots f_s^{\zeta_s}),a)=\mathrm{exp}(cf_1^{\delta_1}\ldots f_s^{\delta_s})(f_1^{\tau_1}\ldots f_s^{\tau_s}),a)=\alpha.
$$

\noindent This proves our assertion.


\noindent  From the hypothesis, 
$ R_a(S_k,\lbrace f_1,\dots,f_s\rbrace)=0$ for all $k\in\lbrace 1,\dots,m\rbrace$. It follows from Theorem \ref{division} that
we can write  $S_k$ as $S_k=\sum_{\underline{\beta}^k}c_{\underline{\beta}^k} f_1^{\beta^k_1}\cdots f_s^{\beta^k_s}$ 
with $\displaystyle{\mathrm{max}_{\beta_k, c_{\beta^k}\not=0}\mathrm{exp}(f_1^{\beta^k_1}\cdots f_s^{\beta^k_s},a)= \mathrm{exp}(S_k,a)}$. Replacing  
$\displaystyle{\sum_{i=1}^lc_{\underline{\theta}^i}f_1^{\theta^i_1}\cdots  f_s^{\theta^i_s}}$ by 

$$
\displaystyle{\sum_{k=1}^m\lambda_k(f_1,\dots,f_s)\sum_{\underline{\beta}^k}c_{\underline{\beta}^k} f_1^{\beta^k_1}\cdots}f_s^{\beta^k_s}
$$

\noindent  in 
the expression of $R$, we can rewrite $R$ as $R=\sum_{\underline{\theta}'}c_{\underline{\theta}'}f_1^{\theta'_1}\cdots f_s^{\theta'_s}
$  with $\alpha_1=\mathrm {max}_{\theta', c_{\theta'}\not=0}\mathrm{exp}(f_1^{\theta'_1}\cdots$\\ $ f_s^{\theta'_s},a) <_a\alpha$. Then 
we restart with this expression. We construct this way an infinite sequence $\alpha >_a\alpha_1>_a\dots >_a\mathrm {exp}(R,a) $. 
This contradicts Lemma \ref{Dixon2}. $\blacksquare$
\end{proof}

\medskip

\noindent The characterization given in Proposition \ref{S-polynomials} suggests an algorithm that constructs, 
starting with a set of generators of ${\bf A}$, an $a$-canonical basis of ${\bf A}$. More precisely we have the following.






\medskip

\noindent {\bf Algorithm}. Let ${\bf A}={\mathbb K}[\![f_1,\dots,f_s]\!]$ and let $a\in U_1^n$. Suppose that $\mathrm{exp}({\bf A})$ is finitely generated and let $\lbrace \bar{S}_1,\dots,\bar{S}_m\rbrace$ 
be a set of generators of the map $\phi$ of Lemma \ref{binomial}. Let $S_i=\bar{S}_i(f_1,\dots,f_s)$ for all $i\in\lbrace 1,\dots,m\rbrace$.

\medskip

\begin{enumerate}
    \item If $R_a(S_i,\lbrace f_1,\dots,f_s\rbrace)=0$ for all $i\in\lbrace 1,\dots,m\rbrace$, then $\lbrace f_1,\dots,f_s\rbrace$ is an $a$-canonical basis of ${\bf A}$.
    
    \item If $R_a(S_i,\lbrace f_1,\dots,f_s\rbrace)\not=0$ for some $i\in\lbrace 1,\dots,m\rbrace$, then 
we add the set  of all the nonzero remainders and we restart with the new system of generators, say $\lbrace f_1,\dots,f_s,f_{s+1},\dots,f_{s+s_1}\rbrace$. Note that in this case, 
we have
 $\langle \mathrm {exp}(f_1,a),\dots,\mathrm {exp}(f_s,a)\rangle \subset \langle \mathrm {exp}(f_1,a),\dots,
\mathrm {exp}(f_{s+s_1},a)\rangle \subseteq \mathrm {exp}({\bf A},a)$. Let us prove that this process will stop.
By hypothesis, $\mathrm{exp}({\bf A},a)$ is finitely generated.  Let $\lbrace \gamma_1,\dots,\gamma_r\rbrace$ be a minimal set of generators 
of $\mathrm{exp}({\bf A},a)$ and suppose that $\nu(\underline{x}^{\gamma_1},a)\leq \cdots \leq \nu(\underline{x}^{\gamma_r},a)$.   Suppose that the process above gives infinitely many 
new elements $f_{s+k},k\geq 1$. Then there exists   $t\geq 1$ such that  $\nu(f_{s+k},a)>\nu(\underline{x}^{\gamma_r},a)$ for all $k> t$  (in fact, the set of $\nu(\underline{x}^{\alpha},a), \alpha\in{\mathbb N}^n$ is finite). We claim that $\lbrace f_1,\dots,f_{s+t}\rbrace$ is an $a$-canonical basis of ${\bf A}$. Suppose otherwise,  then $\gamma_j\notin \lbrace \mathrm{exp}(f_1,a),\dots,\mathrm{exp}(f_{s+t},a)\rbrace$ for some $j\in \lbrace 1,\dots,r\rbrace$, and let $f\in{\bf A}$
 such that $\mathrm{exp}(f,a)=\gamma_j$.  As ${\bf A}={\mathbb K}[\![f_1,\dots,f_s,f_{s+1},\dots,f_{s+t}]\!]$, we can write 
 $f=\sum_{\alpha}c_{\alpha} f_1^{\alpha_1}\cdots f_{s+t}^{\alpha_{s+t}}$. We shall use a similar argument as in Proposition \ref{S-polynomials}. Let 
$\alpha_0=\mathrm{max}_{\alpha, c_{\alpha}\not=0}\mathrm{exp}(f_1^{\alpha_1}\dots f_{s+t}^{\alpha_{s+t}},a)$  and recall that this maximum is well defined. Let $\lbrace \alpha^1,\dots,\alpha^p\rbrace$ be the set of elements 
such that $\alpha_0=\mathrm{exp}(f_1^{\alpha^k_1}\dots f_{s+t}^{\alpha^k_{s+t}},a)$ for all $k\in\lbrace 1,\dots,p\rbrace$.
 We have $\gamma_j\leq_a \alpha_0$. If $\gamma_j=\alpha_0=\mathrm{exp}(f_1^{\alpha^k_1}\dots f_{s+t}^{\alpha^k_{s+t}},a)$, then it follows from the minimality
of $\gamma_j$ that $\gamma_j=\mathrm{exp}(f_l,a)$ for some $l\in\lbrace 1,\dots,s+t\rbrace$. This contradicts the hypothesis.  Whence $\gamma_j<_a\alpha_0$, and $\displaystyle{\sum_{k=1}^{p}c_{\alpha^k}\mathrm{M}(f_1^{\alpha^k_1}\dots f_{s+t}^{\alpha^k_{s+t}}, a)=0}$.  Let $\lbrace \hat{S}_1,\dots,\hat{S}_q\rbrace$ be
 the set of $S$-polynomials associated with the set $\lbrace f_1,\dots,f_{s+t}\rbrace$. As in Proposition \ref{S-polynomials}, we can rewrite $f$ as follows

 $$
 f=\sum_{k=1}^q\lambda_k(f_1,\dots,f_{s+t})\hat{S}_k+f^1
 $$

\noindent with $f^1\in{\bf A}$, and  $\alpha_0>_a\mathrm{max}(\mathrm{exp}(f^1,a), \mathrm{exp}(\sum_{k=1}^q\lambda_k(f_1,\dots,f_{s+t})\hat{S}_k,a)$. By 
Theorem \ref{division},
for all $k$, there exist $F_k, R_k\in{\bf A}$ such that $\hat{S}_k=F_k+R_k$ with $\mathrm{exp}(\hat{S}_k,a)=\mathrm{exp}(F_k,a)$ if $F_k\not=0$, and
either $R_k=0$ or $R_k=f_{s+l}$ for some $l>t$, in particular 
$\nu(R_k)>\nu(\underline{x}^{\gamma_r},a)$. If we write explicitly $F_k$ as in Theorem \ref{division}, we get 

$$
f=\sum_{\alpha'} c_{\alpha'} f_1^{\alpha'_1}\cdots f_{s+t}^{\alpha'_{t}}+\sum_k R_k
$$

\noindent and if $\alpha_1=\mathrm{max}_{\alpha',c_{\alpha'}\not=0}\mathrm{exp}(f_1^{\alpha'_1}\cdots f_{s+t}^{\alpha'_{s+t}},a)$, then
 $\alpha_0>_a\alpha_1>_a\mathrm{exp}(f,a)$. Moreover,
$\nu(R_k,a)>\nu(\underline{x}^{\gamma_r},a)>\nu(f,a)$. Then we restart with this new expression. We construct this way an infinite increasing sequence 
$\alpha_0>_a\alpha_1>_a\cdots>_a \mathrm{exp}(f,a)=\gamma_j$. This  contradicts Lemma \ref{Dixon2}.
\end{enumerate}

\section{Finiteness criterion of canonical bases}

\medskip

\noindent Let the notations be as in Section 1. In particular ${\bf A}={\mathbb K}[\![f_1,\dots,f_s]\!]$ with $\lbrace f_1,\dots,f_s\rbrace \subseteq \CC$.  Let $a\in U^n$. In the following we shall give some
criteria for $\mathrm{exp}({\bf A},a)$ to be finitely generated (or equivalently $\M({\bf A},a)$ is a finitely generated ${\mathbb K}$-algebra). Note first that $\CC$ being an ${\bf A}$-module, the quotient $\CC/{\bf A}$ is  an ${\bf A}$-module, which is also a 
${\mathbb K}$-vector space. The same holds for  $\CC/\mathrm{M}({\bf A},a)$.  With these notations we have the following.

\begin{lema}\label{vectorspace} Let $a\in U^n$ and suppose that $\mathrm{exp}({\bf A},a)$ is finitely generated. Let $\lbrace f_1,\dots,f_r\rbrace$ be the $a$-reduced
 canonical basis
of ${\bf A}$. The map 
$$
\theta: \CC/{\bf A}\to \CC/\mathrm{M}({\bf A},a)
$$
\noindent defined by $\theta(\overline{f})=\overline{R_a(f,\lbrace f_1,\ldots,f_s\rbrace)}$ is an isomorphism of ${\mathbb K}$-vector spaces.
\end{lema}

\begin{proof} It is easy to verify that $\theta$ is well defined and also that it is a linear map. Clearly $\theta$ is surjective, and it follows from Proposition \ref{criterion} that $\mathrm{Ker}(\theta)=0$. This proves our assertion. $\blacksquare$

\end{proof}

\medskip

\noindent Let the notations be as above, and let $a\in U^n$. Let $i\in\lbrace 1,\dots,n\rbrace$, and define   $T_i=\mathrm{exp}_i({\bf A},a)$
 to be the set of $s \in{\mathbb N}$ such that $x_i^{s}\in\M({\bf A},a)$. Clearly $0\in T_i$, and $T_i$ is a subsemigroup of ${\mathbb N}$. We have the following result which will be used in Section 3.

\begin{proposicion}\label{nsg} Let $a\in U^n$ and let  the notations be as above. If $\mathrm{rank}_{\mathbb K}\CC/\mathrm{M}({\bf A},a)<+\infty$, 
then the following conditions hold.

\begin{enumerate}
\item For all $i\in\lbrace 1,\dots,n\rbrace$, $T_i$ is a numerical semigroup of ${\mathbb N}$, i.e. ${\mathbb N}\setminus T_i$ is a finite set.

\item We have the inclusion $T_1\times \cdots \times T_n\subseteq \mathrm{exp}({\bf A},a)$.

\item  $\mathrm{exp}({\bf A},a)$ is finitely generated, i.e. $e({\bf A},a)<+\infty$.
\end{enumerate}
\end{proposicion}
\begin{proof} 1. If  $\mathrm{rank}_{\mathbb K}\CC/\mathrm{M}({\bf A},a)<+\infty$, then 
$\mathrm{Card}({\mathbb N}^n\setminus\mathrm{exp}({\bf A},a))<+\infty$ (where Card stands for the cardinality). But 
$\mathrm{Card}({\mathbb N}\setminus T_i)\leq \mathrm{Card}({\mathbb N}^n\setminus\mathrm{exp}({\bf A},a))$ for all $i\in\lbrace 1,\dots,n\rbrace$, hence 
$\mathrm{Card}({\mathbb N}\setminus T_i)<+\infty$. 

\medskip

\noindent 2. Obvious. 

\medskip

\noindent 3. See \cite{CGU}, Proposition 2.3. $\blacksquare$
\end{proof}

\medskip

\noindent As as corollary we get the following.

\begin{proposicion}\label{finite-finite} With the notations above, if $\mathrm{rank}_{\mathbb K}\CC/\mathrm{M}({\bf A},a_0)<+\infty$
 for some $a_0\in U^n$, then $\mathrm{rank}_{\mathbb K}\CC/\mathrm{M}({\bf A},a)<+\infty$ for all $a\in U^n$. In particular
 $\mathrm{exp}({\bf A}, a)$ is finitely generated for all $a\in U^n$.
\end{proposicion}
\begin{proof} It follows from Lemma \ref{vectorspace} that $\mathrm{rank}_{\mathbb K}\CC/{\bf A}<+\infty$. Whence, by the same lemma,
 $\mathrm{rank}_{\mathbb K}\CC/\mathrm{M}({\bf A},a)<+\infty$ for all $a\in U^n$. In particular, by  Proposition \ref{nsg},
  $\mathrm{exp}({\bf A}, a)$ is finitely generated for all $a\in U^n$.$\blacksquare$
\end{proof}

\medskip

\medskip

\noindent Next we recall a finiteness criterion given in \cite{rs}.

\begin{proposicion} \label{integral}(see \cite{rs}, Proposition 4.7. and 4.9.) Let ${\bf A}$ be as above and let $a\in U^n$. If for all $i\in\lbrace 1,\dots,n\rbrace$, there exists  $\alpha_i\in{\mathbb N}\setminus\lbrace 0\rbrace$ such that $x_i^{\alpha_i}\in\M({\bf A},a)$, then  $\mathrm{exp}({\bf A}, a)$ is finitely generated. 

\end{proposicion}

\begin{proof} The hypothesis is equivalent to saying that ${\bf F}$ is integral over $\M({\bf A},a)$ (see \cite{rs}, 4.9.). The result follows by applying \cite{am}, Proposition 7.8.   to the inclusions ${\mathbb K}\subseteq \M({\bf A},a)\subseteq  \CC$. $\blacksquare$

\end{proof}

\medskip

\noindent Let the notations be as above, and suppose that ${\bf F}$ is integral over $\M({\bf A},a)$. For all $i\in\lbrace 1,\dots,n\rbrace$, let $\alpha_i\in{\mathbb N}\setminus\lbrace 0\rbrace$ such that $x_i^{\alpha_i}\in\M({\bf A},a)$.  Then the cone $C=\lbrace  \sum_{i=1}^n\lambda_i\alpha_i \mid \lambda_i\in {\mathbb R}_+\rbrace$ is nothing but ${\mathbb R}_+^n$, and every element of $\mathrm{exp}({\bf A},a)$
is in this cone.  Thinking of this approach, Proposition \ref{integral} can be adapted to a more general setting. Recall first the following lemma.

\begin{lema} (see \cite{sc}, for example)\label{hb} Let $\lbrace \alpha_1,\ldots,\alpha_s\rbrace$ be a subset of ${\mathbb N}^n$ and let $C=\lbrace  \sum_{i=1}^n\lambda_i\alpha_i \mid \lambda_i\in {\mathbb R}_+\rbrace$ be the cone generated by  $\lbrace \alpha_1,\ldots,\alpha_s\rbrace$. There exists a subset $\lbrace \beta_1,\dots,\beta_r\rbrace$ of $C\cap {\mathbb N}^n$ such that for all $\beta\in C\cap {\mathbb N}^n$, there exist $\lambda_1,\dots,\lambda_s\in {\mathbb N}$, and $j\in\lbrace 1,\dots,r\rbrace$, such that $\beta=\sum_{i=1}^s\lambda_i\alpha_i+\beta_j$.
\end{lema}

\begin{proof} Let $P=\lbrace \sum_{i=1}^s\lambda_ia_i\mid 0\leq \lambda_i<1$ for all $1\leq i\leq s\rbrace$, and let 
$\lbrace \beta_1,\dots,\beta_r\rbrace=P\cap{\mathbb N}^n$. Let $\beta\in C\cap{\mathbb N}^n$. We have 

$$
\beta-\sum_{i=1}^s\left\lfloor \lambda_i\right\rfloor\alpha_i=\sum_{i=1}^s\{\lambda_i\}\alpha_i\in P
$$

\noindent where $\left\lfloor x\right\rfloor$ denotes the integer part of $x$, and $\{x\}$ its fractional part. As $\sum_{i=1}^s\{\lambda_i\}\alpha_i\in P\cap{\mathbb N}^n$, this finishes the proof. $\blacksquare$
\end{proof}

\begin{proposicion}\label{cone} Let ${\bf A}={\mathbb K}[\![f_1,\dots,f_s]\!]$ be as above and let $a\in U^n$. For all $i\in \lbrace 1,\dots,s\rbrace$, let 
$\M(f_i,a)=c_{\alpha_i}\underline{x}^{\alpha_i}$. 
Let $C=\lbrace  \sum_{i=1}^s\lambda_i\alpha_i \mid \lambda_i\in {\mathbb R}_+\rbrace$ be the cone generated by $\lbrace \alpha_1,\dots,\alpha_s\rbrace$, and 
let ${\bf F}_C=\lbrace f\in \CC \mid \mathrm{Supp}(f)\subseteq C\rbrace$ be the ring of polynomials whose supports are in $C$.  If $\M({\bf A},a)\subseteq \CC_C$, then
$\mathrm{exp}({\bf A},a)$ is finitely generated.
\end{proposicion}

\begin{proof} Let the notations be as in Lemma \ref{hb}. In particular $\lbrace \beta_1,\dots,\beta_r\rbrace=P\cap{\mathbb N}^n$. It follows from the proof of the same lemma that 
$\CC_C={\mathbb K}[\![\underline{x}^{\beta_1},\dots,\underline{x}^{\beta_r},\underline{x}^{\alpha_1},\dots,\underline{x}^{\alpha_s}]\!]$, and if
 ${\bf B}={\mathbb K}[\![\underline{x}^{\alpha_1},\dots,\underline{x}^{\alpha_s}]\!]$, then  for all $\underline{x}^{\beta}\in \CC_C$, $\underline{x}^{\beta}\in \underline{x}^{\beta_i}{\bf B}$
 for some $i\in \lbrace 1,\dots,r\rbrace$.  In particular
$\CC_C$ is a finite $\M({\bf A},a)$-module, since it is generated by $\lbrace \underline{x}^{\beta_1},\dots,\underline{x}^{\beta_r}\rbrace$, hence   it is integral over  $\M({\bf A},a)$, and the result follows by applying \cite{am}, Proposition 7.8. to
the inclusions ${\mathbb K}\subseteq \M({\bf A},a)\subseteq \CC_C$. $\blacksquare$
\end{proof}
\medskip

\noindent The following lemma gives a criterion for $\mathrm{exp}({\bf A},a), a\in U^n$, to be in a fixed cone.

\begin{lema}\label{cone1} Let ${\bf A}={\mathbb K}[\![f_1,\dots,f_s]\!]$ be as above and let $a\in U^n$. For all $i\in \lbrace 1,\cdots,s\rbrace$, let 
$\M(f_i,a)=c_{\alpha_i}\underline{x}^{\alpha_i}$. 
Let $C$ be the cone generated by $\lbrace \underline{x}^{\alpha_1},\dots,\underline{x}^{\alpha_s}\rbrace$. If $\mathrm{Supp}(f_i,a)\subseteq C$ for 
all $i\in\lbrace 1,\dots,s\rbrace$, 
then $\mathrm{exp}({\bf A},a)\subseteq C$. In particular, by Proposition \ref{cone},  $\mathrm{exp}({\bf A},a)$ is finitely generated. 

\end{lema}

\begin{proof} Let $f\in{\bf A}$, and write $f=\sum_{\alpha}c_{\alpha}f_1^{\alpha_1}\dots f_s^{\alpha_s}$. For all $\alpha$,  $\mathrm{Supp}(f_1^{\alpha_1}\dots f_s^{\alpha_s})\subseteq C$. Then
$\mathrm{exp}(f,a)\in C$. This finishes the proof. $\blacksquare$

\end{proof}

\medskip

\noindent In the following we shall give some relevant examples (the well ordering here is the lexicographical order).

\begin{example} Let ${\bf A}={\mathbb K}[\![x+y,xy]\!]$ and let $a=(1,2)$.  Then $\mathrm{M}(x+y,a)=x$, and $\M(xy,a)=xy$. Hence
$\lbrace x+y,xy\rbrace$ is the $a$-reduced canonical basis of ${\bf A}$. If $C$ is the cone generated by $(1,0),(1,1)$, then $\mathrm{Supp}(x+y)\not\subset C$. This
proves that the converse of Lemma \ref{cone1} is not true.
\end{example}

\begin{example} (see \cite{rs}) \label{notfinite} Let ${\bf A}={\mathbb K}[\![x+y,xy,xy^2]\!]$ and let $a=(1,2)$. We have 
$\M(x+y,a)=x, \M(xy,a)=xy$, and $\M(xy^2,a)=xy^2$. The kernel of the map:

$$
\phi:{\mathbb K}[X,Y,Z]\to {\mathbb K}[x,y], \phi(X)=x,\phi(Y)=xy,\phi(Z)=xy^2
$$

\noindent is generated by $\bar{S_1}=XZ-Y^2$. We have $S=(x+y)xy^2-x^2y^2=-xy^3=R_a(-xy^3,\lbrace x+y,xy,xy^2\rbrace)$. Hence
$xy^3$ is a new element of the  $a$-reduced canonical basis of ${\bf A}$. If we restart with the representation ${\bf A}={\mathbb K}[\![x+y,xy,xy^2,xy^3]\!]$, then a 
new element, $xy^4$, will be added to the $a$-canonical basis of ${\bf A}$. In fact, $xy^n$ is an element of the  $a$-reduced canonical basis of ${\bf A}$ 
for all $n\geq 1$. In particular the $a$-reduced canonical basis of ${\bf A}$ is infinite. In fact, we can verify that $\mathrm{exp}({\bf A},a)$ is not finitely generated for all $a\in U^n$.
Note that if $\M(x+y,a)=x$ (resp. $\M(x+y,a)=y$), then $(0,1)$ (resp. $(1,0)$) is not in the cone $C$ generated by $(1,0),(1,1),(1,2)$ (resp. $(0,1),(1,1),(1,2)$).

\end{example} 

\begin{example} Let ${\bf A}={\mathbb K}[\![x,xy+y^2,x^2y]\!]$. If $a=(2,1)$, then
$\M(x,a)=x, \M(xy+y^2,a)=y^2$, and $\M(x^2y,a)=x^2y$. The cone $C$ generated by $(1,0),(0,2),(2,1)$ is ${\mathbb R}_+^2$, and the 
hypotheses of Lemma \ref{cone1} are satisfied, hence $\mathrm{exp}({\bf A},a)$ is finitely generated. We  can verify
 that $\lbrace x, xy+y^2, x^2y\rbrace$ is the $a$-reduced canonical basis. In particular $\mathrm{exp}({\bf A},a)=\langle (1,0),(0,2),(2,1)\rangle$.\\
Let $a=(1,2)$. We have $\M(x,a)=x, \M(xy+y^2,a)=xy$, and $\M(x^2y,a)=x^2y$. We can verify that $xy^n$ is an element of the  $a$-reduced canonical basis of ${\bf A}$ for all $n\geq 0$. In particular the $a$-canonical basis of ${\bf A}$ is infinite. Note that in this case, the hypotheses of Lemma \ref{cone1} are not satisfied since $\mathrm{Supp}(xy+y^2)$ is not in the cone generated by 
$(1,0),(1,1),(2,1)$. 

\noindent It is worth noticing here that in both cases, ${\mathrm rank}_{\mathbb K}{\bf F}/{\bf A}$ is not finite (in the first case, $y^{2k+1}\notin{\bf A}$ for all $k\in{\mathbb N}$, and in the second case, $y^k\notin{\bf A}$ for all $k\in{\mathbb N}$ (see Proposition \ref{finite-finite}). 
\end{example}

\begin{example} \label{finitedimension}Let ${\bf A}={\mathbb K}[\![x,xy+y^2,y^3,x^2y]\!]$. For all $a\in U^n$,  
$\M(x,a)=x, \M(y^3,a)=y^3, \M(x^2y,a)=x^2y$, and $\M(xy+y^2,a)$ is either $xy$ or $y^2$, depending on  $(1,0) >_a (0,1)$ or $(0,1) >_a(1,0)$.

\begin{itemize}

\item If $(0,1) >_a(1,0)$, then  we can verify that $\lbrace x, y^2+xy, y^3,x^2y\rbrace$ is the $a$-reduced canonical basis of $\bf A$. In particular 
$\mathrm{exp}({\bf A},a)=\langle (1,0), (0,2),(0,3),(2,1)\rangle$, and ${\mathbb N}^2\setminus\mathrm{exp}({\bf A},a)=\lbrace (0,1),(1,1)\rbrace$.

\item If   $(1,0) >_a (0,1)$, then  we can verify that $\lbrace x,xy+y^2, xy^2,y^3,y^4,y^5\rbrace$ is the $a$-reduced canonical basis of $\bf A$. In particular 
$\mathrm{exp}({\bf A},a)=\langle (1,0), (1,1),(1,2), (0,3),(0,4),(0,5)\rangle$, and ${\mathbb N}^2\setminus\mathrm{exp}({\bf A},a)=\lbrace (0,1),(0,2)\rbrace$.  
\end{itemize} 
\end{example}

\begin{example} 1. (see \cite {agm}) Suppose that $n=1$, i.e. ${\bf A}={\mathbb K}[\![f_1(t),\dots,f_s(t)]\!]\subseteq {\mathbb K}[\![t]\!]$. The parametrization
 ($X_1=f_1(t),\dots,X_s=f_s(t)$) 
represents the expansion of a curve in ${\mathbb K}^s$ near the origin. In this case, $\mathrm{exp}({\bf A},a)$ does not 
depend on $a\in U^n$. If $s=2$ and ${\mathbb N}\setminus \mathrm{exp}({\bf A}, a)<+\infty$, then 
$\mathrm{exp}({\bf A},a)$ is a free numerical semigroup (see \cite{ag} for the definition), and the arithmetic of this semigroup contains a lot of information about the singularity of the curve at the origin.

\medskip

\noindent 2. Let $f(X_1,\dots,X_n,Y)\in{\mathbb K}[\![X_1,\dots,X_n]\!][Y]$ and suppose that 
$f$ has a parametrization of the form $X_1=t_1^{e_1},\dots, X_n=t_n^{e_n}, Y=Y(t_1,\cdots,t_n)\in{\mathbb K}[\![t_1,\dots,t_n]\!]$ (for 
instance, this is true if $f$ is a quasi-ordinary polynomial, i.e. the $Y$-resultant of $f$  with its $Y$-derivative $f_Y$ is of the 
form $X_1^{N_1}\cdots X_s^{N_s}(c+\phi(X_1,\dots,X_s))$ with $c\in{\mathbb K}^*$ and $\phi(0,\dots,0)=0$). We have 
$\dfrac{{\mathbb K}[\![X_1,\dots,X_n]\!][Y]}{f}\simeq {\bf A}={\mathbb K}[\![t_1^{e_1},\dots,t_n^{e_n},Y(t_1,\dots,t_n)]\!]$. In this case, 
$(e_1,0,\dots,0),\dots,(0,\dots,0,e_n)$ belong to $\mathrm{exp}({\bf A},a)$ for all $a\in U^n$. Moreover, ${\rm exp}({\bf A}, a)$  is a free 
 affine semigroup in the sense of \cite{a}.  

\end{example}


\section{A finiteness Theorem}

\noindent Let the notations be as in Section 1. In particular ${\bf A}={\mathbb K}[\![f_1,\dots,f_s]\!]$ with $\lbrace f_1,\dots,f_s\rbrace \subseteq \CC$. The aim of this
 section is to prove, under some additional hypotheses, that the set $\mathrm{M}({\bf A})=\lbrace \mathrm{M}({\bf A},a) \mid a\in U^n\rbrace$ (resp. the set 
$\mathrm{I}({\bf A})=\lbrace \mathrm {in}(A,a) \mid a\in U^n\rbrace$) is finite. Recall that, given a nonzero element $f\in\CC$ and $a\in U^n$, we have 
$\mathrm{M}(f,a)=c_{\mathrm{exp}(f,a)}\underline{x}^{\mathrm{exp}(f,a)}$. In particular, for all $a\in U^n$, we have 
$\mathrm{M}({\bf A},a)={\mathbb K}[\![\mathrm{M}(f_1,a),\dots,\mathrm{M}(f_s,a)]\!]$ if and only if 
$\mathrm{exp}({\bf A},a)=\langle \mathrm{exp}(f_1,a),\dots,\mathrm{exp}(f_s,a)\rangle$. \\
We 
shall first consider the case when ${\bf A}={\mathbb K}[\![f]\!]$. Then we 
prove some preliminary results which will  also be used later in the paper. We start with the following definition:

\begin{definicion}\label{newton} Let $f=\sum_{\alpha}c_{\alpha}\underline{x}^{\alpha}$ be a nonzero element of ${\mathbb K}[\![x_1,\dots,x_n]\!]$ and let
 $E=\bigcup_{\alpha\in \mathrm {Supp}(f)}(\alpha+{\mathbb N}^n)$. We 
have $E+{\mathbb N}^n\subseteq E$, and consequently, by Dickson's Lemma, there exists a unique finite set of 
$E$, say $\lbrace \alpha_1,\dots,\alpha_r\rbrace$, such that $E=\bigcup_{i=1}^r \alpha_i+{\mathbb N}^n$. We set $\mathrm{N}(f)=\lbrace \alpha_1,\dots,\alpha_r\rbrace$, and 
$
f_N=\sum_{i=1}^rc_{\alpha_i}\underline{x}^{\alpha_i}.
$

\end{definicion}

\begin{lema}\label{onegenerator} Let $f$ be a nonzero element of ${\mathbb K}[\![x_1,\dots,x_n]\!]$. The set $\lbrace \mathrm {M}(f,a) \mid a\in U^n\rbrace$
 $\mathrm{(}$resp. $\lbrace \mathrm {in}(f,a) \mid a\in U^n\rbrace$$\mathrm{)}$ is finite.
\end{lema}

\begin{proof} Write $f=\sum_{\alpha} c_{\alpha}{\underline{x}}^{\alpha}$ and let the notations be as in Definition \ref{newton}.  Given $a\in U^n$, it follows from the definition 
of $\mathrm{N}(f)$ that for all $\alpha\in\mathrm{Supp}(f)$, there exists $\alpha_k\in\mathrm{N}(f)$ and
$\beta\in {\mathbb N}^n$ such that $\alpha=\alpha_k+\beta$. In particular $\nu(\underline{x}^{\alpha}, a)\geq \nu(\underline{x}^{\alpha_k},a)$ with inequality if $\alpha\not=\alpha_k$.
 Whence $\mathrm{in}(f,a)=\mathrm{in}(f_N,a)$ and $\mathrm{exp}(f,a)=\mathrm{exp}(f_N,a)$. As $f_N$ is a polynomial, the result follows immediately. $\blacksquare$
\end{proof}

\medskip


\medskip

\begin{lema}\label{rationalone} Let $f=\sum_{\alpha} c_{\alpha}x^{\alpha}$ be a nonzero element of ${\mathbb K}[\![x_1,\dots,x_n]\!]$ and let $a\in U^n$. There 
exists $b\in U_1^n$ such that $\mathrm{in}(f,a)=\mathrm{in}(f,b)$ and $\mathrm{exp}(f,a)=\mathrm{exp}(f,b)$.  
\end{lema}

\begin{proof} Let the notations be as in Definition \ref{newton}. In particular $f_N=\sum_{i=1}^rc_{\alpha_i}\underline{x}^{\alpha_i}$. Let 
$T=\mathrm{Supp}(\mathrm{in}(f,a))$ and
let $d$ be the greatest number of  elements of $T$ that are linearly independent.  Assume, without loss of generality, that 
$T=\lbrace \alpha_{1},\dots,\alpha_d,\dots,
\alpha_{l}\rbrace$, where $\alpha_1,\dots,\alpha_d$ are linearly independent.  We have 
$\nu=\nu(f,a)=\nu(\underline{x}^{\alpha_{1}},a)=\cdots=\nu(\underline{x}^{\alpha_{l}},a)<\nu(\underline{x}^{\alpha_k},a)$ for all $k\geq l+1$.
 Let $0<t<\mathrm{min}_{k\geq l+1}(\nu(\underline{x}^{\alpha_k},a)-\nu)$. The main idea is that we can move the hyperplane  in such a way that the configuration of the points 
of $T$ does not change. More precisely we shall prove the existence of  $\epsilon\in {\mathbb R}_+^n$  such that the following conditions hold.
\begin{enumerate}

\item $a+\epsilon\in U_1^n$.

\item $\epsilon \cdot (\alpha_1-\alpha_{j})=0$ for all $j\in\lbrace 2,\dots,l\rbrace$ (hence $\nu(\underline{x}^{\alpha_1},a+\epsilon)=\cdots=\nu(\underline{x}^{\alpha_l},a+\epsilon)$).

\item $\epsilon\cdot (\alpha_1-\alpha_k)<t$ for all  $k\geq l+1$ (hence $\nu(\underline{x}^{\alpha_1},a+\epsilon)<\nu(\underline{x}^{\alpha_k},a+\epsilon)$ for all $k\geq l+1$).

\end{enumerate}


\noindent Let 
$L: {\mathbb R}^n\to {\mathbb R}^{l-1}, L(X)=(X\cdot (\alpha_1-\alpha_2),\dots, X\cdot (\alpha_1-\alpha_l))$. As the dimension of $\mathrm{Ker}(L)$ is greater than or equal to $1$ ($a\in \mathrm{Ker}(L)$), there exists $\epsilon=(\epsilon_1,\dots,\epsilon_n)\in  \mathrm{Ker}(L)\cap {\mathbb R}_+^n$, relatively small, such that $b=a+\epsilon\in U_1^n$, and
we can choose $\epsilon$ such that condition 3. is satisfied (note that there are infinitely many choices of such an $\epsilon$).  It follows that $\mathrm{in}(f,a)=\mathrm{in}(f,b)$ and $\mathrm{exp}(f,a)=\mathrm{exp}(f,b)$.  This proves our assertion.$\blacksquare$ 
\end{proof}

\begin{nota}\label{rationality} The proof of Lemma \ref{rationalone} can be generalized to a set of elements. More precisely, given $f_1,\dots,f_r\in{\mathbb K}[\![x_1,\dots,x_n]\!]$ and $a\in U^n$, there exists $b\in U_1^n$ such that $\mathrm{in}(f_i,a)=\mathrm{in}(f_i,b)$ for all $i\in\lbrace 1,\dots,r\rbrace$.$\blacksquare$
\end{nota}

\medskip

\noindent Next we prove that if $\mathrm{M}(A,a)\not=\mathrm{M}(A,b)$, then $\mathrm{M}(A,a)\not\subset\mathrm{M}(A,b)$.

\medskip

\begin{lema} \label{samebasis}Let $a\not=b$ be two elements of $U^n$. Assume that $\mathrm{M}({\bf A}, a)$
is finitely generated and  let $\lbrace g_1,\dots,g_r\rbrace$ be the $a$-reduced 
canonical basis of ${\bf A}$. If $\M({\bf A}, a)=\M({\bf A},b)$, then 
$\lbrace g_1,\dots,g_r\rbrace$ is also the $b$-reduced canonical basis of ${\bf A}$.
\end{lema}

\begin{proof} Let $i\in\lbrace 1,\dots,r\rbrace$ and write $g_i=\M(g_i,a)+\sum c_{\beta}\underline{x}^{\beta}$ where for 
all $\beta$, if $c_{\beta}\not=0$,  then $\underline{x}^{\beta}\notin \M({\bf A},a)$. Since $\M({\bf A}, a)=\M({\bf A},b)$, 
then $\underline{x}^{\beta}\notin \M({\bf A},b)$ for all $\beta$ such that $c_{\beta}\not=0$.  This implies that $\M(g_i,a)=\M(g_i,b)$.
 Hence $\lbrace g_1,\dots,g_r\rbrace$ is a $b$-canonical basis of ${\bf A}$, and the same argument shows that this basis is also reduced. $\blacksquare$
\end{proof}

\begin{lema} \label{expequal} Let $a\not=b$ be two elements of $U^n$ and assume that  $\mathrm{M}({\bf A}, b)$
is finitely generated. If $\M({\bf A},a)\not=\M({\bf A},b)$, then  $\M({\bf A},a)\not\subset \M({\bf A},b)$.
\end{lema}

\begin{proof} Assume that $\M({\bf A},a)\subset \M({\bf A},b)$ and let 
 $\lbrace g_1,\dots,g_r\rbrace$ be 
the $b$-reduced canonical basis of $\bf A$. By hypothesis, there is 
$i\in\lbrace 1,\dots,r\rbrace$ such that  $\M(g_i,b)\in \M({\bf A},b)\setminus\M({\bf A},a)$.  Write $g_i=\M(g_i,b)+\sum c_{\beta}\underline{x}^{\beta}$.
For all $\beta$, if $c_{\beta}\not=0$, then $\underline{x}^{\beta}\notin \M({\bf A},b)$, hence $\underline{x}^{\beta}\notin \M({\bf A},a)$. This
 implies that for all $\alpha\in\mathrm{Supp}(g_i), \alpha\notin\mathrm{exp}({\bf A},a)$. This is a contradiction because $g_i\in \bf A$. $\blacksquare$
\end{proof}

\medskip

\noindent In the following we shall prove that it is enough to consider elements in $U_1^n$.

\begin{lema}\label{rational} If $a\in U^n\setminus U_1^n$ and  $\mathrm{M}({\bf A}, a)$
is finitely generated,  then there exists $b\in U_1^n$ such that $\M({\bf A},a)=\M({\bf A},b)$.
\end{lema} 

\begin{proof} Let 
 $\lbrace g_1,\dots,g_r\rbrace$ be 
the $a$-reduced canonical basis of $\bf A$. As $r<+\infty$, we can find,  using Remark \ref{rationality}, $b\in U_1^n$
 such that  $\M(g_i,a)=\M(g_i,b)$ for all $i\in\lbrace 1,\dots,r\rbrace$. It follows that  $\M({\bf A},a)\subseteq \M({\bf A},b)$, whence, by Lemma 
\ref{expequal}, $\M({\bf A},a)=\M({\bf A},b)$. $\blacksquare$
\end{proof}

\medskip

\noindent Next we shall generalize to initial forms the result of Lemma \ref{expequal}. Recall first that given  a subalgebra ${\bf B}$ of $\CC$ and  
$a\in U^n$, ${\bf B}$ is said to be  $a$-homogeneous if ${\bf B}={\mathbb K}[\![G_1,\dots,G_{k},\ldots ]\!]$ where $G_i$ is 
$a$-homogeneous for all $i\geq 1$ (see Definition \ref{Homogeneous}). With these notations we have the following.


\begin{lema} \label{biho} Let  $a\in U_1^n$ and ${\bf B}={\mathbb K}[\![G_1,\dots,G_{s}]\!]$ be an $a$-homogeneous algebra, where we suppose that 
$G_i$ is $a$-homogeneous for all $i\in\lbrace 1,\dots,s\rbrace$. Let $F$ be a non zero element of ${\bf F}$ and let 
$F=\sum_{i}F_{d_i}, d_1<d_2<\cdots$ be the decomposition of
$F$ into $a$-homogeneous components. If $F\in {\bf B}$, then $F_{d_i}\in{\bf B}$ for all $i\geq 1$.

\end{lema}

\begin{proof} Write $F=\sum c_{\alpha}\prod_{i=1}^sG_i^{\alpha_i}$ and let $d=\mathrm{min}_{\alpha}\nu(\prod_{i=1}^sG_i^{\alpha_i},a)$. 
Let also $G=\sum_{d=\nu(\prod_{i=1}^sG_i^{\alpha_i},a)} c_{\alpha}$\\$\prod_{i=1}^sG_i^{\alpha_i}$. We have 
$d_1\geq d$ and $d_1>d$  if and only if $G=0$. In this case we remove $G$ and we restart with $F-G$. Finally we may
assume that $d=d_1$, in particular $F_{d_1}=G\in {\bf B}$. Now we restart with $F-F_{d_1}$. Our assertion follows by an easy induction argument. $\blacksquare$

\end{proof}

\begin{lema}\label{equalin} Let $a\not=b$ be two elements of $U_1^n$ and assume that $\mathrm{M}({\bf A}, a)$
is finitely generated. If $\mathrm{in}({\bf A},a)\not=\mathrm{in}({\bf A},b)$, then
  $\mathrm{in}({\bf A},a)\not\subset \mathrm{in}({\bf A},b)$

\end{lema}

\begin{proof} Suppose that $\mathrm{ in}({\bf A},a)\subseteq \mathrm{ in}({\bf A},b)$, and let us prove that $\mathrm{in}({\bf A},a)=\mathrm{in}({\bf A},b)$.
As $\mathrm{exp}({\bf A},c)=\mathrm{exp}(\mathrm{in}({\bf A},c))$ for all $c\in U^n$, it follows that   $\mathrm{exp}({\bf A},a)\subseteq\mathrm{exp}({\bf A},b)$, 
hence, by Lemma  \ref{expequal}, 
$\mathrm{exp}({\bf A},a)=\mathrm{exp}({\bf A},b)$. Let  $\lbrace g_1,\dots,g_r\rbrace$ be the $a$-reduced canonical basis of ${\bf A}$. Then, by Lemma \ref{samebasis},
$\lbrace g_1,\dots,g_r\rbrace$ is also the $b$-reduced canonical basis of ${\bf A}$. In particular
 $\mathrm{in}({\bf A},a)={\mathbb K}[\![\mathrm{in}(g_1,a),\dots,\mathrm{in}(g_r,a)]\!]$ (resp.
$\mathrm{in}({\bf A},b)={\mathbb K}[\![\mathrm{in}(g_1,b),\dots,\mathrm{in}(g_r,b)]\!]$).
Let us prove that for all $i\in\lbrace 1,\ldots,r\rbrace, \mathrm{in}(g_i,a)=\mathrm{in}(g_i,b)$. Let $G= \mathrm{in}(g_i,a)$, and 
let $G=G_{d_1}+\cdots+G_{d_t}, d_1<\cdots<d_t$,  be the decomposition of 
$G$ into $b$-homogeneous elements. If $t>1$, then $G_{d_k}\in\mathrm{in}({\bf A},b)$ 
for all $k\in \lbrace 2,\dots,t\rbrace$, hence $\mathrm{exp}(G_{d_k},b)\in\mathrm{exp}({\bf A},b)$. But  $ \mathrm{Supp}(G_{d_2}+\cdots+G_{d_t})\cap \mathrm{exp}({\bf A},b)=\emptyset$. This
is a contradiction. Hence $t=1$ and  $\mathrm{in}(g_i,a)=G_{d_1}=\mathrm{in}(g_i,b)$. This  finishes the proof. $\blacksquare$
\end{proof}

 \medskip
 
\noindent We can now state and prove the finiteness theorem. We shall first prove the following proposition.

\begin{proposicion}\label{finite} Let ${\bf A}={\mathbb K}[\![f_1,\dots,f_s]\!]$ and assume that for all 
$a\in U^n$, $\mathrm{M}({\bf A}, a)$ is finitely generated. Assume that
$\mathrm{M}({\bf A})=\lbrace \mathrm{M}({\bf A},a)\mid a\in U^n\rbrace$ is infinite. Given $e\in{\mathbb N}\setminus\lbrace 0\rbrace$, there exists
 $b\in U^n$ such that  $e({\bf A},b)\geq e$.
\end{proposicion}

\begin{proof} Suppose that $\M({\bf A})$ is an infinite set, and let $E=\lbrace a_1,\dots,a_k,\dots\rbrace\in U^n$ such that $\M({\bf A},a_i)\not=\M({\bf A},a_j)$ for all $a_i\not=a_j$. By Lemma \ref{rational}, we may assume that $E\subseteq U_1^n$. Let $T=\lbrace m_1,\dots,m_s\rbrace$ be the set of monomials such that for all $i\in\lbrace 1,\dots,s\rbrace$ and for all $a_k\in E$, $m_i$ is a minimal generator of $\M({\bf A},a_k)$ (the set $T$ is either empty or a finite set). Set $m_i=c_{\alpha_i}\underline{x}^{\alpha_i}$ and let $J_1={\mathbb K}[\![m_1,\dots,m_s]\!]$. Let
$$
\bar{\triangle}=\lbrace \alpha\in{\mathbb N}^n \mid \underline{x}^{\alpha}\notin J_1\rbrace.
$$
\noindent The set $\lbrace f\in {\bf A}\mid \mathrm{Supp}(f)\subseteq \bar{\triangle}\rbrace$ is not empty, otherwise for all $i\in {\mathbb N}, \M({\bf A},a_i)=J_1$, which contradicts the hypothesis. Next we shall use the following
notation: if $m=c_{\alpha}\underline{x}^{\alpha}$ is a monomial of ${\bf  F}$, then we set $\mid m\mid=\sum_{i=1}^n\alpha_i$. Let $G_1$ be the set of
 monomials $m=c\underline{x}^{\alpha}$ in ${\bf F}$ such that the following conditions hold.
\begin{enumerate}
\item There exists $f\in{\bf A}$ such that $\alpha\in\mathrm{Supp}(f)$, and 
$\mathrm{Supp}(f)\subseteq \bar{\triangle}$.
\item There exists an infinite set $\tilde{E}_1\subseteq  E$ such that for all $a\in \tilde{E}_1$, $\mathrm{exp}(f,a)=\alpha$.
\end{enumerate}
\noindent As $\lbrace f\in {\bf A}\mid \mathrm{Supp}(f)\subseteq \bar{\triangle}\rbrace$ is not empty, Lemma \ref{onegenerator} implies that $G_1$ is not empty (in fact, given $f\in{\bf A}$, the set $\lbrace \mathrm{exp}(f,a)\mid a\in U^n\rbrace$ is finite by the same lemma). Let $m'_{s+1}$ be a monomial of
$G_1$ with $\mid m'_{s+1}\mid$ minimal (such a monomial is not necessarily unique). If for all $a\in \tilde{E}_1$, $m'_{s+1}$ is a minimal generator of $\mathrm{M}({\bf A}, a)$, then we set $m'_{s+1}=m_{s+1}, T_1=T\cup\lbrace m_{s+1}\rbrace, E_1=\tilde{E}_1$, and we restart with $T_1,E_1$. Suppose that $m'_{s+1}$ is not a minimal generator of $\mathrm{M}({\bf A}, a)$ for some  $a\in \tilde{E}_1$. Then $m'_{s+1}$ is a product of  minimal generators of $\mathrm{M}({\bf A}, a)$. Every such generator $\underline{x}^{\gamma}$ satisfies $\mid \gamma\mid < \mid m'_{s+1}\mid$, whence the set of these generators is  finite. Let $\gamma\notin\lbrace \alpha_1,\cdots,\alpha_s\rbrace$ be such a generator. There are two possible cases.

\begin{enumerate}

\item There is an infinite subset $\hat{E}_1\subseteq \tilde{E}_1$ such that $\gamma$ is a minimal generator 
 of $\mathrm{exp}({\bf A},a), a\in\hat{E}_1$. In this case we add $\underline{x}^{\gamma}$ to the set $T$.

\item The element $\gamma$ is a minimal generator of at most a finite set of $\mathrm{exp}({\bf A},a), a\in\hat{E}_1$. In this case we remove this set from $\tilde{E}_1$.

\end{enumerate}

\noindent If we do this operation with the set of all minimal generators $\underline{x}^{\gamma}, \gamma\notin\lbrace \alpha_1,\cdots,\alpha_s\rbrace$, we obtain an infinite set $E_1\subseteq E$, and a set $T=\lbrace m_1,\cdots,m_s,m_{s+1},\cdots,m_{s+s_1}\rbrace$ such that for all $i\in\lbrace 1,\cdots,s+s_1\rbrace$, and for all $a\in E_1$, $m_i$ is a minimal generator of $\mathrm{M}({\bf A}, a)$. Then we restart with $T_1$ and $E_1$.  As $e$ is an integer, we get this way  
 $b\in U^n$ such that $\mathrm{M}({\bf A},b)$ 
 is generated by at least $e$ elements. This proves our assertion. $\blacksquare$
\end{proof}
\medskip
\noindent Next we prove that if $\M({\bf A})$ is finite, then so is ${\mathrm I}({\bf A})$. 



\begin{lema} \label{M-I} Let the notations be as in Proposition \ref{finite}. If $\M({\bf A})$ is a finite set, then $\mathrm{I}({\bf A})$ is also a finite set.
\end{lema}

\begin{proof} Assume that  ${\mathrm I}({\bf A})$ is an infinite set. Then there exists an infinite subset $E_1=\lbrace a_1,a_2,\dots \rbrace$ such that for all $i\not=j, \mathrm{I}({\bf A},a_i)\not=\mathrm{I}({\bf A},a_j)$, and for all $k\geq 1$, $\mathrm{M}({\bf A},a_k)=\mathrm{M}({\bf A},a_1)$. Let
$\lbrace g_1,\dots,g_s\rbrace$ be the $a_1$-reduced canonical basis of ${\bf A}$. Then $\lbrace g_1,\dots,g_s\rbrace$ is the $a_k$-reduced canonical basis of ${\bf A}$ for all $k\geq 2$. In particular
$\mathrm{in}({\bf A},a_k)={\mathbb K}[\![\mathrm{in}(g_1,a_k),\dots,\mathrm{in}(g_s,a_k)]\!]$ for all $k\geq 1$. But Lemma \ref{onegenerator} implies that 
the set $\lbrace \mathrm{in}(g_i,a)\mid a\in E_1\rbrace$ is finite for all $i\in\lbrace 1,\dots,s\rbrace$. This is impossible because otherwise we  would have  infinitely many pairwise distinct sets inside the finite set $\cup_{i=1}^s\lbrace \mathrm{in}(g_i,a_k)\mid a_k\in E_1\rbrace$. $\blacksquare$

\end{proof}

\begin{teorema}\label{finiteness} Let $\lbrace f_1,\dots,f_s\rbrace$ be a set of nonzero elements of $\CC$ and let ${\bf A}={\mathbb K}[\![f_1,\dots,f_s]\!]$. Assume that there exists $n_0\in{\mathbb N}$ such that $e({\bf A},a)\leq n_0$ for all $a\in U^n$. The set
$\mathrm {M}({\bf A})=\lbrace \mathrm{M}({\bf A},a)\mid a\in U^n\rbrace$ (hence $\mathrm{exp}({\bf A})=\lbrace \mathrm{exp}({\bf A},a)\mid a\in U^n\rbrace$) is finite. In particular,
the set $\mathrm{I} ({\bf A})=\lbrace \mathrm {in}({\bf A},a)\mid a\in U^n\rbrace$ is finite.
 \end{teorema}

\begin{proof} By Lemma \ref{M-I}, we only need to prove that $\M(A)$ is a finite set, but this follows immediately from Proposition \ref{finite}. $\blacksquare$
\end{proof}


\begin{definicion} With the notations above, suppose that $\mathrm{M}({\bf A})$ is a finite set. The set $\lbrace g_1,\dots,g_r\rbrace$ of ${\bf A}$, which is an $a$-canonical basis of ${\bf A}$ for all $a\in U^n$, is called the universal canonical basis of ${\bf A}$.
 
\end{definicion}

\begin{nota} 1. We do not have a proof for Theorems \ref{finiteness}  for general subalgebras. We think however that the result of this theorem is true.

\medskip

\noindent 2.  If ${\bf A}={\mathbb K}[f_1,\dots,f_s]$ is a subalgebra of ${\bf P}={\mathbb  K}[x_1,\dots,x_n]$, then 
 Theorem \ref{finiteness}  remains valid when we vary $a\in {\mathbb R}^n_+$
 (in \cite{j} the author uses a similar argument as in Proposition \ref{finite}, however, the argument does not seem to suffice 
 without an extra hypothesis. Note that in this case, $\mathrm{in}(f,a)$ is a polynomial for all $a\in{\mathbb R}^n_+$ and for all $f\in \DD$). 
 \end{nota}

\medskip

\noindent As a corollary of Proposition \ref{finite-finite} we get the following.

\begin{corolario} With the notations above, if $d=\mathrm{rank}_{\mathbb K}\CC/\mathrm{M}({\bf A},a_0)<+\infty$ for some $a_0\in U^n$,
 then $e({\bf A},a)\leq 2^dn+2(2^d-1)$ for all 
$a\in U^n$. In particular $\M({\bf A})$ is a finite set.
\end{corolario}

\begin{proof}  We shall use Proposition 2.3. in \cite{CGU}. If $d=0$, then the standard basis  $\lbrace e_1,\dots,e_n\rbrace$ generates 
$\mathrm{exp}({\bf A},a_0)$.  Suppose that $d>0$ and  let $h$ be a maximal element of ${\mathbb N}^n\setminus \mathrm{exp}({\bf A},a_0)$ with respect to the natural partial order on ${\mathbb N}^n$. Then $S=\mathrm{exp}({\bf A},a_0)\cup\lbrace h\rbrace$ is an affine semigroup, and 
${\mathbb N}^n\setminus S$ has cardinality $d-1$. By induction we have $\mathrm{e}(S)\leq 2^{d-1}n+2(2^{d-1}-1)$. If $\lbrace v_1,\dots,v_s\rbrace$ is a minimal set of generators of $S$, then $\lbrace v_1,\dots,v_s,v_1+h,\dots,v_s+h,2h,3h\rbrace$ is a minimal set of generators of 
$\mathrm{exp}({\bf A},a_0)$. In particular $e({\bf A},a_0)=2e({\bf A},a_0)+2\leq 2(2^{d-1}n+2(2^{d-1}-1))+2=2^dn+2(2^d-2)+2=2^dn+2(2^d-1)$. This proves our assertion. As  $\mathrm{rank}_{\mathbb K}\CC/\mathrm{M}({\bf A},a)=d$ for all $a\in U^n$, this implies that $e({\bf A},a)\leq 2^dn+2(2^d-1)$, and by Theorem \ref{finite}, the set $\M({\bf A})$ is a finite set. $\blacksquare$
\end{proof}

\section{The Newton fan}

\medskip

\noindent Let the notations and hypotheses be as in Sections 1 to 3. In particular,  $\lbrace f_1,\dots,f_s\rbrace$ is a set of 
nonzero elements of $\CC$ and  ${\bf A}={\mathbb K}[\![f_1,\dots,f_s]\!]$. Assume that for all $a\in U^n$, $\mathrm{exp}({\bf A},a)$ 
is finitely generated, and that $\M({\bf A})=\lbrace \M({\bf A},a) \mid a\in U^n\rbrace$ is a finite set. We have the following.
\medskip

\begin{teorema}\label{fan} There exists a partition $\mathcal{P}$ of $U^n$ into convex rational polyhedral cones such that for all $\sigma\in\mathcal{P}$, $\mathrm{exp}({\bf A},a)$ and ${\mathrm in}({\bf A},a)$ do not depend on $a\in \sigma$.

\end{teorema}

\noindent In order to prove Theorem \ref{fan} we start by fixing some notations.  Let $S$ be a finitely generated  subsemigroup of ${\mathbb N}^n$ and let 

$$
E_S=\lbrace a\in U^n\mid \mathrm{exp}({\bf A},a)=S\rbrace.
$$
Let $a\in E_S$ and let $\lbrace g_1,\dots,g_r\rbrace$ be the $a$-reduced canonical basis of ${\bf A}$. By Lemma \ref{samebasis},
 $\lbrace g_1,\dots,g_r\rbrace$ is also the $b$-reduced canonical basis of ${\bf A}$ for all $b\in E_S$. Denote by $\sim$ the
 equivalence relation on $U^n$ defined from $\lbrace g_1,\dots,g_r\rbrace$ by

$$
a\sim b \quad \text{if}\quad  \mathrm{in}(g_i,a)=\mathrm{in}(g_i,b) \quad \text {for  all}\quad i\in\lbrace 1,\dots,r\rbrace.
$$

\begin{proposicion} The relation $\sim$ defines on $U^n$ a finite partition, denoted $\mathcal{P_S}$, into convex rational polyhedral cones
and $E_S$ is a union of a part of these cones.

\end{proposicion}

\begin{proof} Let $c,d\in U^n$ such that $c\sim d$ and let $e\in [c, d]$. Let
$\theta \in [0, 1]$ such that $e= \theta c + (1-\theta)d$. We have $e\in U^n$, and  in$(g_i,e)=\mathrm{in}(g_i,c)=\mathrm{in}(g_i,d)$
by an immediate verification. Moreover, $c\sim t\cdot c$ for all $c\in U^n$ and $t>0$. Therefore 
the equivalence classes are convex rational polyhedral cones (the rationality results from Lemma \ref{rational}). On the 
other hand, if $c\sim d$ and $c\in E_S$, then $d\in E_S$ (in fact, for all $i\in \lbrace 1,\dots,r\rbrace$, $\mathrm{in}(g_i,c)=\mathrm{in}(g_i,d)$, whence $\mathrm{M}(g_i,c)=\mathrm{M}(g_i,d)$, in particular $\mathrm{M}({\bf A},c)\subseteq\mathrm{M}({\bf A},d)$, and by Lemma \ref{expequal},  
$\mathrm{M}({\bf A},c)=\mathrm{M}({\bf A},d)$). This proves that $E_S$ is a union of classes of $\sim$, the number of classes
being finite because we have $r$ elements $g_1,\dots,g_r$, and by Lemma \ref{onegenerator},  the set of $\mathrm{in}(g_i,a), a\in U^n$ is finite for all $i\in\lbrace 1,\dots,r\rbrace$. $\blacksquare$ \end{proof}

\medskip 
\noindent Next we  prove that $E_S$ is a convex set.

\begin{lema}\label{convex}  $E_S$ is a convex set: if $a\not=b\in E_S$, then $[a,b]\subseteq E_S$.
\end{lema}
\begin{proof} Let $a,b\in E_S$ and let $\lambda\in]0,1[$. Let $\lbrace g_1,\dots,g_r\rbrace$ be the $a$ (whence the $b$)
 reduced canonical basis of ${\bf A}$. Let $i\in\lbrace 1,\dots,r\rbrace$. We have  $\mathrm{M}(g_i,a)=\mathrm{M}(g_i,b)$ (otherwise, as  
 $\lbrace g_1,\dots,g_r\rbrace$  is a $b$-reduced caconical basis, $\mathrm{exp}(g_i,a)=\mathrm{exp}(g_i-\mathrm{M}(g_i,b),a)\notin \mathrm{exp}({\bf A},b)=\mathrm{exp}(g_i,a)$, which is a contradiction). Let $M={\rm in}(g_i,a)$. Write
 $M=M_1+\cdots+M_t$ where $M_k$ is $b$-homogeneous for all $k\in\lbrace 1,\dots,t\rbrace$ and $\nu(M_1,b)<\nu(M_k,b)$ for 
all $k\in\lbrace 2,\dots,t\rbrace$. In particular $\mathrm{exp}(g_i,a)=\mathrm{exp}(M_1,a)=\mathrm{exp}(M_1, b)={\rm exp}(g_i,b)$. We have $\nu(g_i,a)=\nu(M_1,a)=\nu(M_k,a)$ and  $\nu(M_1,b)<\nu(M_k,b)$ for all $k\in\lbrace 2,\dots,t\rbrace$. By the properties of the inner product we have $\nu(g_i,\theta a)=\nu(M_1,\theta a)=\nu(M_k,\theta a)$ and  $\nu(M_1,(1-\theta)b)<\nu(M_k,(1-\theta)b)$ for all $k\in\lbrace 2,\dots,t\rbrace$.This 
implies, by the same properties, that $\nu(g_i,\theta a+(1-\theta)b)=\nu(M_1,\theta a+(1-\theta)b)<\nu(M_k,\theta a+(1-\theta)b))$ for all $k\in\lbrace 2,\dots,t\rbrace$, hence 
$\mathrm{exp}(g_i,\theta a+(1-\theta)b)=\mathrm{exp}(g_i,a)={\rm exp}(g_i,b)$. In particular 
$\mathrm{exp}({\bf A},a)\subseteq \mathrm{exp}({\bf A},  \theta a+(1-\theta)b)$. By Lemma \ref{expequal} we get the equality. This 
proves that $\mathrm{exp}({\bf A},  \theta a+(1-\theta)b)=S$. $\blacksquare$
\end{proof}
\medskip

\noindent {\bf Proof of Theorem \ref{fan}} We define $\mathcal{P}$ in the following way: for each finitely 
generated  semigroup $S$ we consider the restriction ${\mathcal {P}}_S$ on $E_S$ of the above
 partition.  Then ${\mathcal {P}}$ is the finite union of the $\mathcal{P}_S$'s. On each cone of the partition, 
$\mathrm{in}({\bf A},a)$ and $\mathrm{exp}({\bf A},a)$ are fixed. Conversely assume that $\mathrm{in}({\bf A},a)$ and 
$\mathrm{exp}({\bf A},a)$ are fixed  and let $b$ such that $\mathrm{in}({\bf A},a)=\mathrm{in}({\bf A},b)$. By Lemma \ref{equalin}, the 
$a$-reduced canonical basis $\lbrace g_1,\dots,g_r\rbrace$ of ${\bf A}$ is also the $b$-reduced canonical basis of ${\bf A}$. Moreover,  
$\mathrm{in}(g_i,a)=\mathrm{in}(g_i,b)$ and $\mathrm{exp}(g_i,a)= \mathrm{exp}(g_i,b)$ for all $i\in\lbrace 1,\dots,r\rbrace$. By Lemma \ref{convex}, $E_S$ is a convex set. This ends 
the proof of the theorem. $\blacksquare$


\medskip 

\begin{definicion} ${\mathcal P}$ is called the standard fan of ${\bf A}$.
\end{definicion}



\medskip

\noindent  Next we shall characterize open cones with maximal dimension.  

\begin{definicion} \label{multi}Let $a\in U^n$. We say that $\mathrm{in}({\bf A},a)$ is a monomial algebra
 if it is generated by monomials. In this case, by Lemma \ref{biho}, if $g\in \mathrm{in}({\bf A},a)$, then every 
monomial of $g$ is also in $\mathrm{in}({\bf A},a)$. 
\end{definicion} 

\noindent Monomial elements of ${\bf A}$ have the following characterization.

\begin{lema}\label{monomial} Let $g\in{\bf A}$ and let $a\in U^n$. Then  $\mathrm{in}(g,a)$ is a monomial if and only if there exists 
$\epsilon >0$ such that $\mathrm{in}(g,b)=\mathrm{in}(g,a)$ for all $b\in  B(a,\epsilon)$, where $B(a,\epsilon)$ is the
 ball centered at $a$ of radius $\epsilon$.

\end{lema}

\begin{proof} Suppose that  $\mathrm{in}(g,a)$ is a monomial, and  let, with the notations of Lemma \ref{onegenerator}, $N(g)=\lbrace \alpha_1,\dots,\alpha_r\rbrace$, 
and suppose, without loss of generality, that  ${\mathrm M}(g,a)=c_{\alpha_1}\underline{x}^{\alpha_1}$.  We need to prove the existence of an $\epsilon>0$ such that
 for all $b\in U^n$, if $b\in B(a,\epsilon)$,  then $\mathrm{exp}(g,b)=\mathrm{exp}(g,a)=\alpha_1$. We have 
$\nu(\underline{x}^{\alpha_1},a)<\nu(\underline{x}^{\alpha_k},a)$ for all $k\in\lbrace 2,\dots,r\rbrace$. Let 
$-t=\nu(\underline{x}^{\alpha_1},a)-\nu(\underline{x}^{\alpha_2},a)=a\cdot (\alpha_1-\alpha_2)$. Let 
$\underline{\epsilon}\in{\mathbb R}_+^n$ and consider the equation
$$
(a+\underline{\epsilon})\cdot (\alpha_1-\alpha_2)=-t+\underline{\epsilon}\cdot (\alpha_1-\alpha_2).
$$

\noindent We need this quantity to be negative. This is 
equivalent to $\underline{\epsilon}\cdot (\alpha_1-\alpha_2)\leq t$. Let $\theta$ be the angle $(\underline{\epsilon},\alpha_1-\alpha_2)$. As  
$\underline{\epsilon}\cdot (\alpha_1-\alpha_2)=\|\underline{\epsilon}\| \|\alpha_1-\alpha_2\| \mathrm{cos}(\theta)$, the inequality is equivalent to 
$\|\underline{\epsilon}\|\mathrm{cos}(\theta)\leq \dfrac{t}{\|\alpha_1-\alpha_2\|}$. But $\mathrm{cos}(\theta)\leq 1$, whence we choose 
$\underline{\epsilon}$ q
such that $\|\underline{\epsilon}\| \leq  \dfrac{t}{\|\alpha_1-\alpha_2\|}$. Let $\epsilon^2=\mathrm{min}(\epsilon_1,\dots,\epsilon_n)$. For all 
$b\in B(a,\epsilon^2)$, there exists $\underline{\epsilon}\in {\mathbb R}_+^n$ such that $b=a+\underline{\epsilon}$, and by the argument above we have 
$(a+\underline{\epsilon})\cdot (\alpha_1-\alpha_2)<0$. If we repeat this argument with $\alpha_3,\dots,\alpha_r$, and if we take 
$\epsilon=\mathrm{min}(\epsilon^2,\dots,\epsilon^r)$, then we get the result.

\noindent Conversely suppose that $\mathrm{in}(g,b)=\mathrm{in}(g,a)$ for all $b\in  B(a,\epsilon)$ where $\epsilon\in{\mathbb R}_+$.  Write
 $\mathrm{in}(g,a)=\sum_{i=1}^sc_{\beta_i}\underline{x}^{\beta_i}$, and suppose that $s>1$. Using the same argument and notations as above, we can choose an element 
$a+\underline{\epsilon}\in B(a,\epsilon)$ such that 
 $(a+\underline{\epsilon}) \cdot (\beta_1-\beta_2)<0$. This is a contradiction because the
initial form does not depend on the choice of elements in $B(a,\epsilon)$. $\blacksquare$
 \end{proof}
 

\begin{lema} Let $a\in U^n$ and let $\lbrace g_1,\dots,g_r\rbrace$ be the $a$-reduced canonical basis 
of ${\bf A}$. The algebra $\mathrm{in}({\bf A},a)$ is monomial if and only if  $\mathrm{in}(g_i,a)$ is a 
monomial for all $i\in\lbrace 1,\dots,r\rbrace$.
\end{lema}

\begin{proof} We only need to prove the if part. Let $i\in\lbrace 1,\dots,r\rbrace$ and write $\mathrm{in}(g_i,a)=M_1+\cdots+M_t$ with
 $\mathrm{exp}(g_i,a)=\mathrm{exp}(M_1,a)$. Assume that $t>1$. Since $\mathrm{in}({\bf A},a)$ is monomial, we have 
$M_i\in \mathrm{in}({\bf A},a)$ for all $i\in\lbrace 2,\dots,t\rbrace$. But $\lbrace g_1,\dots,g_r\rbrace$ is reduced. This is a contradiction. Hence 
$t=1$ and $\mathrm{in}(g_i,a)$ is a monomial. $\blacksquare$
\end{proof} 

\begin{proposicion} The set of $a\in U^n$ for which $\mathrm{in}({\bf A},a)$ is monomial defines the open cones of dimension $n$ of ${\mathcal{P}}$. 

\end{proposicion}

\begin{proof} Let $a\in U^n$ and suppose that in$({\bf A},a)$ is monomial. Let $\sigma$ be the cone corresponding to $a$ and let 
$\lbrace g_1,\dots,g_r\rbrace$ be the $a$-reduced canonical basis of ${\bf A}$. For all $i\in\lbrace 1,\dots,r\rbrace$, $\mathrm{in}(g_i,a)$ is a
 monomial, hence, by Lemma \ref{monomial}, there exists $\epsilon >0$ such that $\mathrm{in}(g_i,b)=\mathrm{in}(g_i,a)$ for all $b\in B(a,\epsilon)$. This proves, by Lemma \ref{expequal} and Lemma \ref{equalin},  that $\lbrace g_1,\dots,g_r\rbrace$ is also 
the $b$-reduced canonical basis of ${\bf A}$ for all $b\in B(a,\epsilon)$. Hence $B(a,\epsilon)\subseteq \sigma$. Conversely,
let $a$ be an element of an open cone $\sigma$ of dimension $n$ of $\mathcal{P}$, and let $\lbrace g_1,\dots,g_r\rbrace$ be the $a$-reduced canonical basis. For all 
$b$ in a neighbourhood of $a$, we have $\mathrm{in}({\bf A},b)=\mathrm{in}({\bf A},a)$ and $\mathrm{exp}({\bf A},b)=\mathrm{exp}({\bf A},a)$. In particular 
$\lbrace g_1,\dots,g_r\rbrace$ is also the $b$-reduced canonical basis. Moreover, for all $i\in\lbrace 1,\dots,r\rbrace$, $\mathrm{in}(g_i,a)=\mathrm{in}(g_i,b)$ and  
$\mathrm{exp}(g_i,a)=\mathrm{exp}(g_i,b)$. As $\mathrm{in}(g_i,b)$ does not depend on $b$ in a small ball $B(a,\epsilon)$, it results from Lemma
\ref{monomial} that $\mathrm{in}(g_i,a)$ is a monomial, and consequently $\mathrm{in}({\bf A},a)$ is a monomial algebra. $\blacksquare$
\end{proof}

\begin{example}  Let ${\bf A}$ be as in Example \ref{finitedimension}. Then  ${\bf A}={\mathbb K}[\![x,xy+y^2,y^3,x^2y]\!]$, and $\mathrm{exp}({\bf A},a)$ is
either $\langle (1,0), (1,1),(1,2), (0,3),(0,4),(0,5)\rangle$, or $\langle (1,0), (0,2),(0,3),(2,1)\rangle$, depending on  $(1,0) >_a (0,1)$ or $(0,1) >_a(1,0)$. The fan 
associated with ${\bf A}$  is $C_1\cup C_2\cup C_3$ where $C_1$ is the open cone generated by $(1,0),(1,1)$ 
(with $\mathrm{in}({\bf A},a)=\mathrm{M}({\bf A},a)={\mathbb K}[\![x,xy, xy^2,y^3,y^4,y^5]\!]$ for all $a\in C_1$), $C_2$  is the open cone generated by $(0,1),(1,1)$
(with $\mathrm{in}({\bf A},a)=\mathrm{M}({\bf A},a)={\mathbb K}[\![ x, y^2, y^3,x^2y]\!]$ for all $a\in C_2$), and $C_3$ is the 
line generated by $(1,1)$ (with  $\mathrm{in}({\bf A},a)={\mathbb K}[\![x,xy+y^2, xy^2,y^3,y^4,y^5]\!]$ and $\mathrm{M}({\bf A},a)={\mathbb K}[\![x,xy, xy^2,y^3,y^4,y^5]\!]$ for
all $a\in C_3$).  

\medskip

\noindent{\bf Acknowledgements.} The author would like to thank the referee for his many valuable comments on the previous versions of the paper.



\end{example}

\end{document}